\documentclass[11pt]{article}
\usepackage[margin=1.2in]{geometry}

\usepackage[T1]{fontenc}
\usepackage{lmodern}
\usepackage{amsmath,amssymb,amsthm,mathtools}
\usepackage{graphicx}
\usepackage{microtype}
\usepackage{hyperref}

\newtheorem{theorem}{Theorem}[section]
\newtheorem{proposition}[theorem]{Proposition}
\newtheorem{lemma}[theorem]{Lemma}
\newtheorem{corollary}[theorem]{Corollary}

\theoremstyle{definition}
\newtheorem{definition}[theorem]{Definition}

\theoremstyle{remark}
\newtheorem{remark}[theorem]{Remark}

\theoremstyle{example}
\newtheorem{example}[theorem]{Example}

\numberwithin{equation}{section}

\let\Re\relax
\let\Im\relax
\DeclareMathOperator{\Re}{Re}
\DeclareMathOperator{\Im}{Im}
\DeclareMathOperator{\Nrd}{Nrd}
\DeclareMathOperator{\Stab}{Stab}
\DeclareMathOperator{\diag}{diag}
\DeclareMathOperator{\End}{End}

\newcommand{\Gtwo}{{\mathrm{G}_2}}
\DeclareMathOperator{\SU}{SU}
\DeclareMathOperator{\Sp}{Sp}
\newcommand{\su}{\mathfrak{su}}

\newcommand{\avec}{\mathbf{a}}
\newcommand{\svec}{\mathbf{s}}

\newcommand{\inputsection}[1]{%
  \IfFileExists{sections/#1.tex}{\input{sections/#1}}{\typeout{Not yet present: sections/#1.tex}}}

\title{The two-step spectral locus of the octonionic structure operator}
\author{Sergey Grigorian\\School of Mathematical \& Statistical Sciences \\University of Texas Rio Grande Valley\\Edinburg, TX 78539}
\date{October 14, 2026}
\begin{document}
\maketitle

\begin{abstract}
We continue the study of the octonionic spectrum of the
octonionic structure operator $h$ on $W=\mathbb O\otimes_{\mathbb R}V$, where
$V$ is the seven-dimensional representation of $\Gtwo$. The real spectrum and the octonionic one-step spectrum, given by the equation $h(w)=w\lambda$ for real or octonion $\lambda$, have been determined in earlier work. In this paper, we study the two-step
equation $h(w)=(w\lambda_1)\lambda_2$. Due to non-associativity, this does not in general reduce to the one-step problem. An adapted quaternionic frame represents successive right multiplication
by an operator with four real coefficients. The $\SU(2)$ subgroup fixing the frame splits $W$ into
invariant sectors of dimensions $16$, $12$ and $28$. We factor the determinant
of the resulting 4-parameter operator pencil and characterize the pairs for which a nonzero
solution exists. The singular locus of this pencil in $\mathbb R^4$ is the union of an
explicit algebraic set and a round $3$-sphere. The $12$-dimensional sector
carries the one-step problem at the reversed product $\lambda_2\lambda_1$.
The $28$-dimensional sector contributes no additional real singular points;
the fixed sector produces solutions even when neither product lies in the
one-step right spectrum.
\end{abstract}

\section{Introduction}
\label{sec:introduction}

The guiding question in this paper is the following. Given octonions $\lambda_1,\lambda_2$, when does the equation
\begin{equation} \label{eq:def}
 h(w)=(w\lambda_1)\lambda_2,\qquad w\ne0,
\end{equation}
have a solution? Here $h$ is the canonical structure operator on
$W=\mathbb O\otimes_{\mathbb R}V$, where $V$ is the standard real
$7$-dimensional $\Gtwo$-module, and right multiplication acts on each
octonion coefficient of $w\in W$. In this paper, we study this two-step problem, posed in
\cite[p.~53]{grigorian-2026-octonionic}. The operator $h$ is defined naturally by octonion multiplication and is both $G_2$-equivariant and octonion-Hermitian. Among operators of left octonionic multiplication type, these conditions give a real two-parameter family \cite[Remark~3.4, (3.6), p.~13]{grigorian-2026-octonionic}. Its nonzero members, modulo positive scale, form an $S^1$-family. For non-scalar members, the real and one-step spectra and the full two-step coefficient loci are related by invertible affine changes of coordinates, as explained in Remark~\ref{rem:family}; the scalar members are degenerate.

With the associator convention $[a,b,c]=a(bc)-(ab)c$,
\begin{equation*} 
 (w\lambda_1)\lambda_2=w(\lambda_1\lambda_2)-[w,\lambda_1,\lambda_2].
\end{equation*}
Over an associative coefficient algebra the equation collapses to the
one-step problem with right eigenvalue $\lambda_1\lambda_2$.

For octonions, we ask how the associator changes the spectral locus.
By Artin's Theorem, $\lambda_1$ and $\lambda_2$ lie in an associative quaternionic
subalgebra, however, the octonion coefficients of $w$ need not lie in that subalgebra
\cite[Chapter~6 and Sections~8.1--8.5]{conway-2003-quaternions}.
On the quaternionic subalgebra, successive right multiplication agrees with
multiplication by $\lambda_1\lambda_2$; on its orthogonal complement, it
agrees with multiplication by the reversed product $\lambda_2\lambda_1$
(Proposition~\ref{prop:halves}). The operator $h$ can couple these subspaces,
so this description does not reduce the problem to two independent
one-step equations.

Equation~\eqref{eq:def} reduces to the problem of finding the kernel of the operator $h-R_{\lambda_2}R_{\lambda_1}$, where $R$ denotes the octonion right-multiplication operator on $W$. For every pair of octonions $(\lambda_1,\lambda_2)$ this operator, in an adapted orthonormal pair of imaginary units $u,v$, is identified with a member $P_{\avec}$ of the four-coefficient pencil
\begin{equation*}
 P_{\avec}=h-a_0I-a_1R_u-a_2R_v-a_3R_vR_u,\qquad\text{where}\qquad
 R_vR_u:w\longmapsto(wu)v,
\end{equation*}
at a point $\avec=(a_0,a_1,a_2,a_3)\in\mathbb R^4$
(Proposition~\ref{prop:general-reduction}), and every coefficient point $\avec$ is
realized by a pair $(\lambda_1,\lambda_2)$ (Proposition~\ref{prop:realizability}).

Write
$x=a_1^2+a_2^2$, $y=a_3^2$ and $N=a_0^2+x+y$. Then, for a pair,
$a_0=\Re(\lambda_1\lambda_2)$,  $N=|\lambda_1|^2|\lambda_2|^2$, and $y$ is the Gram determinant of the imaginary parts. The reduced point $\avec$ depends on the frame, while $a_0$, $N$ and $y$ are
frame-independent. Our results therefore apply to every pair at its reduced point.

The stabilizer of the quaternionic frame $(u,v,uv)$ is
$K\cong\SU(2)$. Its action decomposes $W$ into three invariant sectors $W_{\mathbf1}\cong{\mathbf1}^{\oplus16}, W_{\mathbf3}={\mathbf3}^{\oplus4}, W_{\mathbf4}={\mathbf4}^{\oplus7}$, 
of real dimensions $16$, $12$ and $28$, respectively (Theorem~\ref{thm:sectors}), and $\mathbf1,\mathbf3,\mathbf4$ refer to real irreducible representations of $\SU(2)$.

Theorem~\ref{thm:two-step-determinant} gives the polynomial identity
\begin{equation*}
 \det\nolimits_{56}P_{\avec}
   =N^6Q_{14}(a_0,x,y)(N+4a_0)^3Q_{12}(a_0,x,y)^2,
\end{equation*}
where $Q_{14},Q_{12}$ are explicit polynomials defined in
Sections~\ref{subsec:fixed-sector-block} and~\ref{subsec:four-block}.
Theorem~\ref{thm:real-locus} identifies the real singular locus as $Z\cup S$,
where
\begin{equation*}
\begin{aligned}
 Z&=\{\avec\in\mathbb R^4:Q_{14}(a_0,a_1^2+a_2^2,a_3^2)=0\},\\
 S&=\{\avec\in\mathbb R^4:(a_0+2)^2+a_1^2+a_2^2+a_3^2=4\}.
\end{aligned}
\end{equation*}
Their form for every pair is given in Corollary~\ref{cor:general-pairs}.
The sphere arises from the reversed-product sector $W_{\mathbf3}$. The quaternionic-type
sector $W_{\mathbf4}$ contributes no additional real singular points, while the fixed
sector $W_{\mathbf1}$ supplies the locus $Z$.

The real and one-step spectra were determined in
\cite{grigorian-2026-octonionic}: the former is finite, and the latter is described in
each complex slice by a quartic and a circle
(Section~\ref{sec:one-step-spectrum}, Figure~\ref{fig:level1-slice-locus}).
To the best of the author's knowledge, the sector determinant formulas and
the resulting characterization of the real two-step locus are new.

Sections~\ref{sec:one-step-spectrum} and~\ref{sec:two-step-pencil} recall
the lower levels and set up the pencil and sectors.
Section~\ref{sec:determinant-locus} proves
Theorems~\ref{thm:two-step-determinant} and~\ref{thm:real-locus}, and
Section~\ref{sec:conclusion} states the remaining questions.
The quaternionic-type coordinate and matrix calculation is given in
Section~\ref{subsec:four-block}. Appendix~\ref{app:su2-bases} supplies
the algebraic rules and branching, and
Appendix~\ref{app:fixed-sector-calculation} gives the fixed-sector Schur
reduction.

\subsection*{Acknowledgments}

Codex and Claude Code were used to organize and verify supporting research materials, to assist with Python scripts that were used for computations, and for assistance with typesetting. ChatGPT was used for mathematical brainstorming and editorial assistance with the exposition. The main ideas are the author's own and all mathematical results have been verified by the author, who takes full responsibility for the content.

\section{Background and the one-step spectrum}
\label{sec:one-step-spectrum}
Let us first introduce the conventions and the notation.
\begin{definition}[Conventions]
\label{def:octonion-conventions}

We use the conventions of \cite[\S2.1, (2.1)--(2.3), p.~6; Remark~5.1, p.~23; (5.4), p.~25]{grigorian-2026-octonionic}.
  Let $\mathbb O$ be the octonions, with conjugation $p\mapsto\bar p$ and inner product $\langle p,r\rangle=\Re(p\bar r)$, and let $g$ be the restriction of this inner product to $\Im\mathbb O$. For $a,b\in\Im\mathbb O$, the cross product $a\times b$ is defined by $ab=-g(a,b)1+a\times b$. The associator $[a,b,c]$ of $a,b,c\in\mathbb O$ is defined by
  \begin{equation}
    \label{eq:associator-rearranged}
    (ab)c=a(bc)-[a,b,c].
  \end{equation}
  Finally, $\Gtwo=\operatorname{Aut}(\mathbb O)$, and for a unit $\hat n\in\Im\mathbb O$ the complex structure $J_{\hat n}$ on $\hat n^\perp\subset\Im\mathbb O$ is $J_{\hat n}(p)=\hat n\times p$.
\end{definition}

\begin{definition}
\label{def:space-w}
Following \cite[\S2.3]{grigorian-2026-octonionic}, fix an oriented Euclidean
$7$-space $V$ and a $\Gtwo$-equivariant isometry
$V\cong\widehat V:=\Im\mathbb O$.
We distinguish corresponding vectors by writing $v\in V$ and
$\hat v\in\widehat V$. Set $W:=\mathbb O\otimes_{\mathbb R}V$,
a real vector space of dimension $56$.

The product metric and diagonal $\Gtwo$-action on $W$ are
\[
\begin{aligned}
  \langle p\otimes v,r\otimes v'\rangle_W
    &=\langle p,r\rangle_{\mathbb O}\,g(v,v'),\\
  \gamma\cdot(p\otimes v)&=\gamma(p)\otimes\gamma(v),
\end{aligned}
\]
where $\Gtwo$ acts on $V$ through the fixed isometry.

All octonion products are formed in the first tensor factor.
For $a\in\mathbb O$, define
\[
  L_a(p\otimes v)=(ap)\otimes v
  \qquad\text{and}\qquad
  R_a(p\otimes v)=(pa)\otimes v.
\]
We write $wa:=R_aw$ for $w\in W$. Thus successive right multiplication
is formed coefficient by coefficient, in the order
$(w\lambda_1)\lambda_2=R_{\lambda_2}R_{\lambda_1}w$.
\end{definition}

Fix corresponding orthonormal bases $\{e_i\}$ of $V$ and
$\{\hat e_i\}$ of $\widehat V$. Write
$w=\sum_j\hat w_j\otimes e_j$, with $\hat w_j\in\mathbb O$;
identity~\eqref{eq:associator-rearranged} applies to each coefficient.
In the expansion $\hat w_j=w_{0j}1+\sum_iw_{ij}\hat e_i$, the first index
of $w_{ij}$ belongs to the octonion factor and the second to $V$.

The imaginary-coefficient part $\widehat V\otimes V\subset W$ is therefore
identified with $V\otimes V$, with the octonion index first.
This identifies tensor subspaces such as $\Lambda^2_7V$ and $S^2_0V$
with subspaces of $W$. As a $\Gtwo$-module,
\[
W\cong\mathbf1\oplus\mathbf7\oplus\mathbf7
       \oplus\mathbf{14}\oplus\mathbf{27},
\]
where bold numerals denote irreducible representations of $\Gtwo$ of the indicated dimensions. The summands $\mathbf1$, $\Lambda^2_7V$,
$\mathbf{14}$, and $\mathbf{27}$ lie in
$\widehat V\otimes V\cong V\otimes V
=S^2V\oplus\Lambda^2V$ and have imaginary
$\mathbb O$-coefficients, whereas $1\otimes V$
has real $\mathbb O$-coefficients.

\begin{definition}
\label{def:operator-h}
 With the bases fixed above, let $w=\sum_j\hat w_j\otimes e_j$ for $\hat w_j\in\mathbb O$, and define, as in \cite[\S3.1, Proposition~3.3 and Remark~3.4]{grigorian-2026-octonionic}, $h\in\End_{\mathbb R}(W)$ by
  \begin{equation}
    \label{eq:h-definition}
    h(w)=\sum_{i,j}\bigl((\hat e_i\hat e_j)\,\hat w_j\bigr)\otimes e_i.
  \end{equation}
  In \eqref{eq:h-definition} the product $h_{ij}:=\hat e_i\hat e_j$ is formed first and then multiplies $\hat w_j$ on the left; the second index of $h_{ij}$ is contracted against the $V$-index of $w$, and the first is the output $V$-index.
\end{definition}

The operator $h$ does not depend on the choice of corresponding orthonormal bases \cite[\S3.1]{grigorian-2026-octonionic}. The octonion product also determines the adjoint and equivariance properties of $h$.


\begin{proposition}[{\cite[Proposition~3.2]{grigorian-2026-octonionic}}]
\label{prop:equivariance}
  The operator $h$ is self-adjoint with respect to $\langle\cdot,\cdot\rangle_W$ and $\Gtwo$-equivariant.
\end{proposition}

\begin{definition}
    \label{def:levels}
We call $h(w)=\mu w$ ($\mu\in\mathbb R$) the \emph{level-0} problem,
$h(w)=w\lambda$ ($\lambda\in\mathbb O$) the \emph{level-1} or \emph{one-step}
problem, and $h(w)=(w\lambda_1)\lambda_2$ ($\lambda_1,\lambda_2\in\mathbb O$)
the \emph{level-2} or \emph{two-step} problem. Hence, the \emph{level-1 right spectrum} of $h$ is $\operatorname{Spec}_{\mathrm{right}}(h)=\{\lambda\in\mathbb O:\ker(h-R_\lambda)\neq0\}$. An octonion $\lambda$ is a \emph{right eigenvalue}, with eigenvector $w\neq0$, when $h(w)=w\lambda$.
\end{definition}

\begin{proposition}[{Level-$0$ real spectrum \cite[Lemma~4.2 and Theorem~4.3]{grigorian-2026-octonionic}}]
\label{prop:real-spectrum}

The action of $h$ on the isotypic components of $W$
is summarized in the following table:
\[
\begin{array}{c|l|c|c}
\Gtwo\text{-type}
  & \text{Realization in }W
  & \text{Eigenvalue(s) of }h
  & \text{Multiplicity}\\
\hline
\mathbf1
  & \mathbb R\sum_i\hat e_i\otimes e_i
  & -7 & 1\\
\mathbf7\oplus\mathbf7
  & (1\otimes V)\oplus\Lambda^2_7V
  & r_\pm=\dfrac{1\pm\sqrt{33}}2
  & 7\text{ each}\\
\mathbf{14}
  & \Lambda^2_{14}V
  & -4 & 14\\
\mathbf{27}
  & S^2_0V
  & 0 & 27
\end{array}
\]
On the $\mathbf7\oplus\mathbf7$ component, $h$ acts
through a $2\times2$ matrix with characteristic
polynomial $\mu^2-\mu-8$. Consequently,
\[
\operatorname{Spec}_{\mathbb R}(h)
=\{-7,-4,0,r_-,r_+\},
\]
with multiplicities $1,14,27,7,7$, respectively.
\end{proposition}

For a unit $n\in V$ with image $\hat n$, the \emph{slice} is $\mathbb C_{\hat n}=\mathbb R\oplus\mathbb R\hat n$, in which each $\lambda$ is uniquely $\mu+Q\hat n$ with $\mu,Q\in\mathbb R$, and the \emph{slice pencil} is $H_{n,Q}=h-Q\,R_{\hat n}$. We take $Q$ to be \emph{signed}: $(\mu,Q,\hat n)$ and $(\mu,-Q,-\hat n)$ represent the same $\lambda$, and every slice locus is even in $Q$. Since $R_\lambda$ is real-linear in $\lambda$ and $\mu$ is central,
$w\lambda=\mu w+QR_{\hat n}w$, so
\[
  h(w)=w\lambda\quad\text{if and only if}\quad H_{n,Q}w=\mu w.
\]

Let $\SU(3)_n=\Stab_{\Gtwo}(n)$. It follows immediately \cite[Proposition~5.4]{grigorian-2026-octonionic} that $R_{\hat n}$, and hence the slice pencil $H_{n,Q}$, is equivariant under $\SU(3)_n \cong \SU(3)$. The pencil is therefore block-diagonal for the $\SU(3)$-decomposition of $W$. Here the bold numerals $\mathbf1,\mathbf6,\mathbf8,\mathbf{12}$ name types of $\SU(3)_n$-modules that correspond to \emph{real} irreducible representations of $\SU(3)$ of the specified dimension.


The four real module types have the following descriptions:
\[
\begin{array}{c|l}
\text{Type} & \text{Description}\\\hline
\mathbf1 & \mathbb Rn\\
\mathbf6 & n^\perp,\quad (n^\perp)_{\mathbb C}\cong\mathbf3\oplus\bar{\mathbf3}\\
\mathbf8 & \su(n^\perp,J_n)\\
\mathbf{12} & \text{Real module with complexification }S^2(\mathbf3)\oplus S^2(\bar{\mathbf3})
\end{array}
\]
Here $\mathbf3$ is the standard complex $\SU(3)_n$-module and
$\bar{\mathbf3}$ its conjugate. The $\Gtwo$ summands restrict as
\[
\begin{aligned}
  \mathbf7&\cong\mathbf1\oplus\mathbf6,\\
  \mathbf{14}&\cong\mathbf8\oplus\mathbf6,\\
  \mathbf{27}&\cong\mathbf1\oplus\mathbf6\oplus\mathbf8\oplus\mathbf{12}.
\end{aligned}
\]
Consequently,
  \begin{equation}
    \label{eq:su3-decomposition}
    W\cong\mathbf1^{\oplus4}\oplus\mathbf6^{\oplus4}\oplus\mathbf8^{\oplus2}\oplus\mathbf{12},
  \end{equation}
  where the two copies of $\mathbf8$ come from $\mathbf{14}$ and $\mathbf{27}$. Each $H_{n,Q}$ preserves every isotypic summand.

The four types are irreducible over $\mathbb R$
\cite[Section~2, pp.~58--60]{moroianu-2008-deformations}.
Their real or complex types are given in
\cite[Section~26.3, Proposition~26.24]{fulton-1991-representation}.
The complex irreducible
modules underlying $\mathbf6$, $\mathbf8$ and $\mathbf{12}$ are described in
\cite[Section~6.1; proof of Proposition~6.17; p.~147 and Section~6.9,
Exercise~8; Section~6.5]{hall-2015-lie}.

\begin{proposition}[The one-step spectrum \cite{grigorian-2026-octonionic}]
\label{prop:level1-spectrum}
Fix a unit $n\in V$. We then have the following.
\begin{enumerate}
\item  On the $\mathbf1^{\oplus4}$ summand, $H_{n,Q}-\mu I $ is a real $4\times4$ operator whose determinant is
  \begin{equation}
    \label{eq:p1}
    p_{\mathbf1}(\mu,Q)=\mu^4+6\mu^3+(2Q^2-15)\mu^2+(6Q^2-56)\mu+Q^4+Q^2 ,
  \end{equation}
  and the block contributes a right eigenvalue $\mu+Q\hat n$ exactly when $p_{\mathbf1}(\mu,Q)=0$.

\item The $\mathbf8^{\oplus2}$ block contributes a right eigenvalue $\mu+Q\hat n$ exactly when
  \begin{equation}
    \label{eq:circle}
    (\mu+2)^2+Q^2=4 .
  \end{equation}

\item For real $Q\neq0$, $H_{n,Q}$ has no real eigenvalue on the $\mathbf{12}$ summand or on the $\mathbf6^{\oplus4}$ block.

\item Let
  \begin{equation}
    \label{eq:slice-locus}
    \mathcal L=\{p_{\mathbf1}=0\}\cup\{(\mu+2)^2+Q^2=4\}\subset\mathbb R^2 ,
  \end{equation}
  where $p_{\mathbf1}$ is the quartic \eqref{eq:p1}. For each unit $n\in V$, the right eigenvalues of $h$ in $\mathbb C_{\hat n}$ are exactly the $\mu+Q\hat n$ with $(\mu,Q)\in\mathcal L$. The real points $Q=0$ of $\mathcal L$ are exactly $\operatorname{Spec}_{\mathbb R}(h)$.

\end{enumerate}
\end{proposition}

The locus $\mathcal L$ is plotted in Figure~\ref{fig:level1-slice-locus}.

\begin{figure}[ht]
  \centering
  \includegraphics[width=0.72\textwidth]{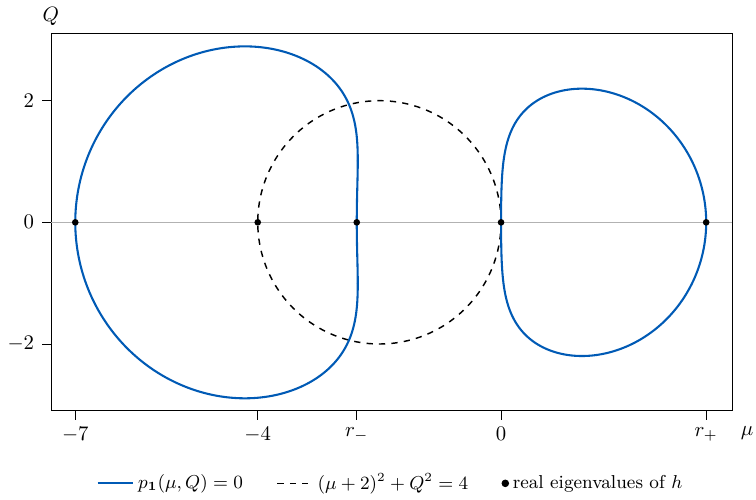}
  \caption{The level-1 slice locus $\mathcal L$ in the $(\mu,Q)$-plane, with equal scales: the quartic $p_{\mathbf1}(\mu,Q)=0$ (solid) and the circle $(\mu+2)^2+Q^2=4$ (dashed). The dots on the line $Q=0$ are the real eigenvalues $-7,-4,0,r_\pm$ of $h$. Since $Q$ is signed, $\mathcal L$ is even in $Q$. Compare \cite[Figure~1]{grigorian-2026-octonionic}, where $r_\pm$ is written $\mu_\pm$.}
  \label{fig:level1-slice-locus}
\end{figure}

\begin{remark}[The family of structure operators]
\label{rem:family}
As shown in \cite[Remark~3.4, (3.6), p.~13]{grigorian-2026-octonionic}, the
$\Gtwo$-equivariant, octonion-Hermitian operators of the
multiplication type considered here form a real two-parameter family, with coefficients
\[
(h_{\alpha,\beta})_{ij}
=-\alpha g_{ij}1+\beta\varphi_{ijk}\hat e_k.
\]
The nonzero members, modulo positive scale, form an $S^1$-family;
one may fix $\alpha^2+\beta^2=1$.
Since $h=h_{1,1}$, we have
\[
h_{\alpha,\beta}=\beta h+(\beta-\alpha)I.
\]
For $\beta\ne0$, the real and one-step right spectra are therefore
obtained from those of $h$ by $\lambda\mapsto\beta\lambda+(\beta-\alpha)$,
with the same eigenspaces. For the two-step coefficient locus,
put $s=\beta-\alpha$. The pencil of Section~\ref{sec:two-step-pencil} satisfies
\[
 h_{\alpha,\beta}-a_0I-a_1R_u-a_2R_v-a_3R_vR_u
 =\beta P_{((a_0-s)/\beta,\,a_1/\beta,\,a_2/\beta,\,a_3/\beta)}.
\]
This invertible affine change preserves kernels, and every coefficient
point is realizable (Proposition~\ref{prop:realizability}). It concerns the full
coefficient locus, rather than a fixed-angle slice or a map of individual
factor pairs. When $\beta=0$, the operator is scalar, $h_{\alpha,0}=-\alpha I$,
and the affine equivalence above does not apply.
\end{remark}

\section{Reduction of the two-step problem}
\label{sec:two-step-pencil}

We reduce the two-step equation to a pencil in four real coefficients and show
that every coefficient point occurs. The adapted frame then describes the
right-multiplication part on two halves of $W$ and splits $W$ into three
invariant sectors.

\begin{definition}
  \label{def:two-step-pencil}
  Let $u,v\in\Im\mathbb O$ be orthonormal. For $\avec=(a_0,a_1,a_2,a_3)\in\mathbb R^4$
  the \emph{two-step pencil} is the four-parameter affine family of operators on $W$
  \begin{equation}
    P_\avec=h-a_0I-a_1R_u-a_2R_v-a_3R_vR_u,\qquad\text{where}\qquad R_vR_u\colon w\mapsto(wu)v .
    \label{eq:two-step-pencil}
  \end{equation}
\end{definition}

Each operator in this pencil is defined for every $\avec\in\mathbb R^4$; statements about $P_\avec$ without
further qualification refer to this ambient coefficient space. The order
$R_vR_u$ is the composition order of Section~\ref{sec:one-step-spectrum}, in
which right multiplications act on the octonion factor of $W$.

\subsection{The frame and arbitrary pairs}
\label{subsec:frame}

The orthonormal pair $(u,v)$ generates the associative quaternion algebra
$\mathbb H=\langle1,u,v,z\rangle$, where $z=uv$
\cite[Section~6.8, Theorem~2]{conway-2003-quaternions}.
Let $\hat A=\Im\mathbb H$ and $\hat B=\mathbb H^\perp$ in the octonion factor,
and let $A,B\subset V$ be their corresponding subspaces. Thus
\begin{equation}
  \mathbb O=\mathbb H\oplus\hat B \qquad \text{and} \qquad V=A\oplus B.
  \label{eq:frame-splittings}
\end{equation}
We use the basis $(\hat\rho_0,\hat\rho_1,\hat\rho_2,\hat\rho_3)=(1,u,v,z)$
of $\mathbb H$. Choose a unit $\ell\in\hat B$ and set
$\hat\xi_\alpha=\hat\rho_\alpha\ell$. The corresponding orthonormal bases are
\[
\begin{array}{c|c}
\text{Octonion factor} & V\text{ factor}\\\hline
\hat A:\ (\hat\rho_1,\hat\rho_2,\hat\rho_3) & A:\ (\rho_1,\rho_2,\rho_3)\\
\hat B:\ (\hat\xi_0,\hat\xi_1,\hat\xi_2,\hat\xi_3) & B:\ (\xi_0,\xi_1,\xi_2,\xi_3)
\end{array}
\]

The splitting of the octonion factor gives two \emph{halves} of $W$:
\begin{equation}
  W_+=\mathbb H\otimes V \qquad \text{and} \qquad W_-=\hat B\otimes V.
  \label{eq:halves}
\end{equation}
Let $\sigma$ be the involution of $\mathbb O$ equal to $1$ on
$\mathbb H$ and to $-1$ on $\hat B$, and let $\Pi_{\hat B}$ be the orthogonal
projection of $\mathbb O$ onto $\hat B$. Then $\sigma=1-2\Pi_{\hat B}$. The orthogonal projections to $W_\pm$ are then $\Pi_\pm$, with $\Pi_-=\Pi_{\hat B}\otimes1$.

The frame stabilizer will also give invariant sectors for the full pencil.
Let $K=\SU(2)=\Stab_{\Gtwo}(u,v)$ be the residual symmetry group of the
frame; it fixes $\mathbb H$ pointwise
\cite[Chapter~6, Appendix~2, p.~81]{conway-2003-quaternions}.
For a $K$-module $U$, write
\[
  U^K=\{s\in U:ks=s\text{ for every }k\in K\}.
\]
The multiplication table of $\mathbb H$ is in
Appendix~\ref{app:su2-bases}, which also describes the action of $K$ on $\hat B$ and the Cayley--Dickson rules that use the basis $\hat\xi_\alpha=\hat\rho_\alpha\ell$.

An arbitrary pair $(\lambda_1,\lambda_2)$ of octonions also determines such a frame by Gram-Schmidt orthogonalization. Take
$u=\Im\lambda_1/|\Im\lambda_1|$ if $\Im\lambda_1\neq0$, and any unit imaginary
octonion otherwise, and write $\Im\lambda_2=\alpha u+\beta v$ with
$\alpha,\beta\in\mathbb R$ and $v\perp u$ a unit imaginary octonion; if
$\Im\lambda_2\in\mathbb Ru$, any such $v$ is allowed and $\beta=0$. Then $(u,v)$
is orthonormal, and
\begin{equation}
  \lambda_1=\mu_1+Q_1u\qquad\text{and}\qquad \lambda_2=\mu_2+\alpha u+\beta v
  \label{eq:general-pair}
\end{equation}
with $\mu_1,Q_1,\mu_2\in\mathbb R$.

Now, for $\avec\in\mathbb R^4$ we define
\begin{equation}
  x=a_1^2+a_2^2,\qquad y=a_3^2\qquad\text{and}\qquad N=a_0^2+x+y,
  \label{eq:invariants}
\end{equation}
and, in the frame $(u,v)$, define the quaternions
\begin{equation}
  \hat q=a_0+a_1u+a_2v+a_3z \qquad \text{and} \qquad \check q=a_0+a_1u+a_2v-a_3z .
  \label{eq:normal-reversed}
\end{equation}

The coefficients of the reduced pencil and their intrinsic meaning are as follows.

\begin{proposition}
  \label{prop:general-reduction}
  Let $(\lambda_1,\lambda_2)$ be a pair of octonions, written in a
  frame $(u,v)$ as in \eqref{eq:general-pair}.

  \begin{enumerate}
    \item On $\mathbb O$ and on $W$,
      \begin{equation*}
        R_{\lambda_2}R_{\lambda_1}
        =a_0+a_1R_u+a_2R_v+a_3R_vR_u,
      \end{equation*}
      where
      \begin{equation}
        \begin{aligned}
          a_0+ia_1&=(\mu_1+iQ_1)(\mu_2+i\alpha),\\
          a_2+ia_3&=\beta(\mu_1+iQ_1).
        \end{aligned}
        \label{eq:reduced-point}
      \end{equation}
      Consequently,
      $h-R_{\lambda_2}R_{\lambda_1}=P_\avec$.
      We call $\avec$ the \emph{reduced point} of the pair
      in the frame $(u,v)$.

    \item Let $\hat q,\check q$ be as in
      \eqref{eq:normal-reversed}, and let $x,y,N$ be as in
      \eqref{eq:invariants}.

      The quaternionic elements $\hat q$ and $\check q$ are
      the products of the original octonions in the two orders:
      \begin{equation*}
        \hat q=\lambda_1\lambda_2
        \qquad\text{and}\qquad
        \check q=\lambda_2\lambda_1.
      \end{equation*}

      Moreover, the invariants admit the following
      intrinsic descriptions:
      \begin{align*}
        a_0
          &=\Re(\lambda_1\lambda_2)
            =\mu_1\mu_2
             -\langle\Im\lambda_1,\Im\lambda_2\rangle,\\
        N
          &=|\lambda_1|^2|\lambda_2|^2,\\
        y
          &=|\Im\lambda_1|^2|\Im\lambda_2|^2
            -\langle\Im\lambda_1,\Im\lambda_2\rangle^2
            =Q_1^2\beta^2,\\
        x
          &=N-a_0^2-y.
      \end{align*}
      Thus $y$ is the Gram determinant of the imaginary
      parts of $\lambda_1$ and $\lambda_2$.

      In particular, $y=a_3^2=0$ if and only if $\Im\lambda_1$ and $\Im\lambda_2$ are linearly dependent. When these equivalent conditions hold, $ R_{\lambda_2}R_{\lambda_1} =R_{\lambda_1\lambda_2}.$
  
    \item Let
      \[
        \mathcal E=\mathrm{span}_{\mathbb R}\{u,v\}
        \qquad\text{and}\qquad
        \Gamma=\{g\in\Gtwo:g\mathcal E=\mathcal E\}.
      \]
      For $g\in\Gamma$, let $\kappa\in O(2)$ be the
      matrix of its action on $\mathcal E$ in the basis $(u,v)$.
      Define
      \[
        \kappa\cdot\avec
        =\bigl(a_0,\,
          \kappa(a_1,a_2)^{\mathsf T},\,
          (\det\kappa)a_3\bigr).
      \]
      Then
      \begin{equation}
        gP_\avec g^{-1}=P_{\kappa\cdot\avec}
        \qquad\text{for }\avec\in\mathbb R^4.
        \label{eq:frame-conjugation}
      \end{equation}
  \end{enumerate}
\end{proposition}

\begin{proof}
  \emph{Part 1.}
  By real linearity of right multiplication,
  \[
    R_{\lambda_1}=\mu_1+Q_1R_u
    \qquad\text{and}\qquad
    R_{\lambda_2}=\mu_2+\alpha R_u+\beta R_v.
  \]
  Right alternativity gives
  $(wu)u=w(uu)=-w$, hence $R_u^2=-1$.
  Expanding the product, we obtain
  \begin{align*}
    R_{\lambda_2}R_{\lambda_1}
      &=(\mu_1\mu_2-Q_1\alpha)
        +(Q_1\mu_2+\mu_1\alpha)R_u\\
      &\quad+\mu_1\beta R_v
        +Q_1\beta R_vR_u.
  \end{align*}
  The coefficients agree with
  \eqref{eq:reduced-point}, proving the stated identity
  on $\mathbb O$. The same identity holds on $W$ because
  right multiplications act on the octonion factor.

  \medskip
  \noindent\emph{Part 2.}
  Since $\lambda_1,\lambda_2$ belong to the associative
  quaternion subalgebra generated by $u$ and $v$,
  we may multiply them using
  \[
    u^2=v^2=-1\qquad\text{and}\qquad uv=z=-vu.
  \]
  This gives
  \begin{align*}
    \lambda_1\lambda_2
      &=(\mu_1+Q_1u)(\mu_2+\alpha u+\beta v)\\
      &=(\mu_1\mu_2-Q_1\alpha)
        +(\mu_1\alpha+Q_1\mu_2)u\\
      &\quad+\mu_1\beta v+Q_1\beta z\\
      &=a_0+a_1u+a_2v+a_3z
       =\hat q.
  \end{align*}
  Reversing the order of multiplication changes only
  the sign of the $z$-component, since $vu=-uv$. Thus
  \begin{equation*}
    \lambda_2\lambda_1
      =a_0+a_1u+a_2v-a_3z
      =\check q.
  \end{equation*}

  The imaginary parts are $Q_1u$ and $\alpha u+\beta v$, with
  inner product $Q_1\alpha$. Thus $a_0=\mu_1\mu_2-Q_1\alpha$
  is $\Re(\lambda_1\lambda_2)$, and their Gram determinant is
  \[
    Q_1^2(\alpha^2+\beta^2)-(Q_1\alpha)^2=Q_1^2\beta^2=a_3^2=y.
  \]
  Norm multiplicativity gives
  $N=|\hat q|^2=|\lambda_1|^2|\lambda_2|^2$, and the definition
  of $N$ gives $x=N-a_0^2-y$.

  Since the inner product on $\Im\mathbb O$ is
  positive definite, the Gram determinant vanishes
  precisely when $\Im\lambda_1$ and $\Im\lambda_2$
  are linearly dependent. This proves the stated
  characterization of $y=0$.

  Finally, if $y=0$, then $a_3=0$, so
  $\hat q=a_0+a_1u+a_2v$. By part~1 and real
  linearity of right multiplication,
  \begin{align*}
    R_{\lambda_2}R_{\lambda_1}
      &=a_0+a_1R_u+a_2R_v\\
      &=R_{a_0+a_1u+a_2v}\\
      &=R_{\hat q}
       =R_{\lambda_1\lambda_2}.
  \end{align*}

  \medskip
  \noindent\emph{Part 3.}
  Let $g\in\Gamma$. Since $h$ is $\Gtwo$-equivariant
  (Proposition~\ref{prop:equivariance}) and
  $gR_pg^{-1}=R_{gp}$ for $p\in\mathbb O$, we have
  \begin{equation*}
    gP_\avec g^{-1}
      =h-a_0-a_1R_{gu}-a_2R_{gv}
       -a_3R_{gv}R_{gu}.
  \end{equation*}

  Right alternativity and its polarization give
  \begin{equation}
    R_u^2=R_v^2=-1
    \qquad\text{and}\qquad
    R_uR_v+R_vR_u
      =-2\langle u,v\rangle=0.
    \label{eq:right-anticommutation}
  \end{equation}
  Since the columns of $\kappa$ are orthonormal,
  these relations imply
  \[
    R_{gv}R_{gu}
      =(\det\kappa)R_vR_u.
  \]
  On the other hand, the linear combination
  $a_1R_{gu}+a_2R_{gv}$ has coefficients
  $\kappa(a_1,a_2)^{\mathsf T}$ in the basis
  $(R_u,R_v)$.

  Thus conjugation by $g$ transforms the coefficients
  precisely according to $\kappa\cdot\avec$, giving
  \[
    gP_\avec g^{-1}=P_{\kappa\cdot\avec}.
  \]
\end{proof}

\begin{proposition}
  \label{prop:realizability}
  For any fixed frame $(u,v)$, the map
  $(\mu_1,Q_1,\mu_2,\alpha,\beta)\mapsto\avec$ is surjective from
  $\mathbb R^5$ onto $\mathbb R^4$.
\end{proposition}

\begin{proof}
  Given $\avec\in\mathbb R^4$, set
  $c_1=a_0+ia_1$ and $c_2=a_2+ia_3$.
  Then \eqref{eq:reduced-point} is satisfied by
  \[
    (\beta,\mu_1+iQ_1,\mu_2+i\alpha)
    =
    \begin{cases}
      (1,c_2,c_1/c_2), & c_2\neq0,\\
      (0,1,c_1), & c_2=0.
    \end{cases}
  \]
\end{proof}

Thus every pair has a reduced point, and every coefficient point occurs. The fourth operator of the pencil, $R_vR_u$, can be evaluated on the frame.

\begin{proposition}
  \label{prop:sign-twist}
  For all $p\in\mathbb O$,
  \begin{equation*}
    (pu)v=\sigma(p)\,z,
  \end{equation*}
  or equivalently, $[p,u,v]=2(\Pi_{\hat B}p)\,z$ with the associator as defined by \eqref{eq:associator-rearranged}. On $W$, we then have
  \begin{equation}
    R_vR_u=R_z\circ(\sigma\otimes1),
    \label{eq:sign-twist}
  \end{equation}
  and consequently, for all $\avec\in\mathbb R^4$,
  \begin{equation*}
    P_\avec=h-a_0I-a_1R_u-a_2R_v-a_3R_z(\sigma\otimes1).
  \end{equation*}
\end{proposition}

\begin{proof}
For \(p\in\mathbb H\), associativity of \(\mathbb H\) gives
\((pu)v=pz\).
For \(p=r\ell\in\hat B\), the Cayley--Dickson rules
\eqref{eq:cayley-dickson-rules}
\cite[Section~2.2]{BaezOcto}, also described in
\cite[Section~6.3 and Chapter~6, Appendix~1]{conway-2003-quaternions}, give
\begin{equation*}
  ((r\ell)u)v
    =((r\bar u)\ell)v
    =((r\bar u)\bar v)\ell
    =(rz)\ell\qquad\text{and}\qquad
  (r\ell)z=(r\bar z)\ell=-(rz)\ell,
\end{equation*}
because \(\bar u\bar v=\overline{vu}=z\); so \((pu)v=-pz\) on \(\hat B\).
Thus
\begin{equation}
\label{eq:sign-twist-associator}
  (pu)v=\sigma(p)z\qquad\text{for every }p\in\mathbb O.
\end{equation}
Since right multiplications act on the octonion factor of
\(W\), \eqref{eq:sign-twist-associator} gives
\begin{equation*}
  R_vR_u=R_z\circ(\sigma\otimes1)
  \quad\text{on }W=\mathbb O\otimes V.
\end{equation*}
The associator form follows from
$\sigma(p)z=pz-2(\Pi_{\hat B}p)z$ and
\eqref{eq:associator-rearranged}. Substitution into the pencil gives its
stated form. In particular, $R_vR_u=R_z$ on $W_+$ and $R_vR_u=-R_z$ on $W_-$.
\end{proof}

\begin{proposition}
  \label{prop:halves}
  Fix a frame $(u,v)$, with associated decomposition
  $W=W_+\oplus W_-$ and projections $\Pi_\pm$.
  \begin{enumerate}
    \item[\textup{(a)}] For every $\avec\in\mathbb R^4$,
      with $\hat q,\check q$ as in \eqref{eq:normal-reversed},
      \begin{equation}
        \begin{aligned}
          P_\avec
            &=h-R_{\hat q}\Pi_+-R_{\check q}\Pi_-.
        \end{aligned}
        \label{eq:halves-identity}
      \end{equation}

    \item[\textup{(b)}] If $\avec$ is the reduced point of a pair
      $(\lambda_1,\lambda_2)$ in this frame, then
      \begin{equation}
        P_\avec
          =h-R_{\lambda_1\lambda_2}\Pi_+
             -R_{\lambda_2\lambda_1}\Pi_-.
        \label{eq:pair-halves}
      \end{equation}
      Equivalently, right multiplication satisfies
      \begin{equation*}
        (w\lambda_1)\lambda_2=
        \begin{cases}
          w(\lambda_1\lambda_2), & w\in W_+,\\
          w(\lambda_2\lambda_1), & w\in W_-.
        \end{cases}
      \end{equation*}
  \end{enumerate}
\end{proposition}

\begin{proof}
  By Proposition~\ref{prop:sign-twist},
  $R_vR_u=\pm R_z$ on $W_\pm$. Thus, by real linearity
  of right multiplication,
  \[
    a_0I+a_1R_u+a_2R_v+a_3R_vR_u
    =R_{a_0+a_1u+a_2v\pm a_3z}
    \quad\text{on }W_\pm.
  \]
  These are precisely $R_{\hat q}$ and $R_{\check q}$,
  proving (a). For (b), Proposition~\ref{prop:general-reduction}
  gives $P_\avec=h-R_{\lambda_2}R_{\lambda_1}$,
  $\hat q=\lambda_1\lambda_2$, and
  $\check q=\lambda_2\lambda_1$.
  Substituting into (a) proves the claim.
\end{proof}

Each half is preserved by right multiplication by elements of $\mathbb H$.
However, $h$ may couple the two halves. To obtain invariant subspaces for
the full pencil, we now decompose $W$ into $K$-isotypic sectors.

\subsection{The \texorpdfstring{$\SU(2)$}{SU(2)} sectors}
\label{subsec:sectors}

The operator $h$ and all four operators of the pencil commute with $K=\SU(2)$, so $W$ splits according to representations of $K$. The action of $K$ on $\hat B$ and the types of its irreducible modules that occur in $W$ are described in Appendix~\ref{app:su2-bases}, equation~\eqref{eq:frame-stabilizer-action} and Lemma~\ref{lem:su2-branching}. The real irreducible types that occur are $\mathbf1$, $\mathbf3$ and
$\mathbf4$, of dimensions $1$, $3$ and $4$, with $\mathbf4$ of quaternionic
type. The subscript of $W_{\mathbf i}$ denotes the dimension of its
irreducible type; the sector itself can have larger dimension.
For $\gamma\in\mathbb H$ and $X\in\su(2)$ let
\begin{equation}
  \begin{aligned}
    \varepsilon_\gamma&=\sum_{\alpha=0}^{3}(\gamma\hat\xi_\alpha)\otimes\xi_\alpha ,\\
    \text{and}\quad w_{X,\gamma}&=\sum_{\alpha=0}^{3}\bigl((X\hat\xi_\alpha)\gamma\bigr)\otimes\xi_\alpha ,
  \end{aligned}
  \label{eq:epsilon}
\end{equation}
where $\varepsilon_\gamma$ does not depend on the orthonormal basis $(\hat\xi_\alpha)$ of $\hat B$, and $X$ acts on $\hat B$ as an element of the Lie algebra of
$\SU(2)\subset\Gtwo$, that is, as a derivation of $\mathbb O$ that vanishes on
$\mathbb H$ and preserves $\hat B$.


\begin{theorem}
  \label{thm:sectors}
  For $\mathbf i\in\{\mathbf1,\mathbf3,\mathbf4\}$,
  let $W_{\mathbf i}$ denote the $K$-isotypic component
  of $W$ of type $\mathbf i$. Then
  \begin{equation}
    W=W_{\mathbf1}\oplus W_{\mathbf3}\oplus W_{\mathbf4}
    \label{eq:sector-decomposition}
  \end{equation}
  is an orthogonal decomposition, with each summand preserved
  by $P_\avec$ for every $\avec\in\mathbb R^4$ and by the
  plane stabilizer $\Gamma$. More precisely:
  \begin{enumerate}
    \item[\textup{(a)}]
      The trivial isotypic component is
      \[
        W_{\mathbf1}=W^K
          =(\mathbb H\otimes A)
            \oplus\{\varepsilon_\gamma:\gamma\in\mathbb H\},
        \qquad\text{and}\qquad
        \dim_{\mathbb R}W_{\mathbf1}=16.
      \]

    \item[\textup{(b)}]
      The component $W_{\mathbf3}$ has real dimension $12$ and
      satisfies
      \[
        W_{\mathbf3}\subset\hat B\otimes B
        \qquad\text{and}\qquad W_{\mathbf3}\subset W_-.
      \]
      The map $X\otimes\gamma\mapsto w_{X,\gamma}$
      defines a $K$-equivariant isomorphism
      \[
        \su(2)\otimes\mathbb H\cong W_{\mathbf3},
      \]
      where $K\cong\SU(2)$ acts by $\operatorname{Ad}$
      on $\su(2)$ and trivially on $\mathbb H$.

    \item[\textup{(c)}]
      The remaining component is
      \[
        W_{\mathbf4}
          =(\mathbb H\otimes B)\oplus(\hat B\otimes A),
        \qquad\text{and}\qquad
        \dim_{\mathbb R}W_{\mathbf4}=28.
      \]

    \item[\textup{(d)}]
      As a $K$-module,
      \[
        W_{\mathbf4}\cong\hat B^{\oplus7}
        \qquad\text{and}\qquad
        \operatorname{End}_K(W_{\mathbf4})
          \cong M_7(\mathbb H),
      \]
      where $\hat B\cong\mathbb H$ is the irreducible
      module of type $\mathbf4$, of quaternionic type.
      Under a $K$-equivariant identification with
      $\hat B^{\oplus7}$, the matrices act on row vectors
      from the right.

  \end{enumerate}
\end{theorem}
\begin{proof}
  The operator $h$ is $\Gtwo$-equivariant
  (Proposition~\ref{prop:equivariance}), while $R_u$ and
  $R_v$ commute with $K$. Thus every $P_\avec$ commutes
  with $K$ and preserves its isotypic components.
  These components are mutually orthogonal since the
  $K$-action is orthogonal. Lemma~\ref{lem:su2-branching} gives
  \[
    W\cong\mathbf1^{\oplus16}\oplus\mathbf3^{\oplus4}
      \oplus\mathbf4^{\oplus7}
  \]
  and locates these isotypic summands in the tensor decomposition.
  Their real dimensions are $16$, $12$ and $28$. We now establish
  the explicit identifications in (a)--(d), using
  the commutant description in Appendix~\ref{app:su2-bases}.

By Lemma~\ref{lem:su2-branching},
\(W_{\mathbf1}=W^K=(\mathbb H\otimes A)\oplus(\hat B\otimes B)^K\) and
\(\dim W^K=16\).
The metric and \(B\cong\hat B\) identify
\((\hat B\otimes B)^K\) with \(\End_K(\hat B)\); this identification
\(\hat B\otimes B=\End(\hat B)\) is basis free, and it sends
\(\varepsilon_\gamma\) of \eqref{eq:epsilon} to
\(\hat\eta\mapsto\sum_{\alpha}\langle\hat\eta,\hat\xi_\alpha\rangle\,
\gamma\hat\xi_\alpha=\gamma\hat\eta\), for every orthonormal basis
\((\hat\xi_\alpha)\) of \(\hat B\).
The commutant described in Appendix~\ref{app:su2-bases} consists of the maps
\(\hat\xi\mapsto\gamma\hat\xi\), whose tensors are the
\(\varepsilon_\gamma\), which therefore do not depend on the choice of basis.

Under the metric identification
\(\hat B\otimes B\cong\End(\hat B)\), the tensor \(w_{X,\gamma}\) corresponds to
\(R_\gamma\circ X|_{\hat B}\), so it does not depend on the choice of
orthonormal basis. We prove the parametrization of \(W_{\mathbf3}\):
the map
\(X\otimes\gamma\mapsto w_{X,\gamma}\) of \eqref{eq:epsilon} is a
\(K\)-equivariant isomorphism
\(\su(2)\otimes\mathbb H\to W_{\mathbf3}\), with \(K\) acting by \(\operatorname{Ad}\) on
\(\su(2)\) and trivially on the label \(\gamma\).

The map is equivariant: since \(g\in K\) is an automorphism of \(\mathbb O\) that fixes \(\gamma\) and
maps orthonormal bases to orthonormal bases,
\(g\cdot w_{X,\gamma}=w_{\operatorname{Ad}_gX,\gamma}\).

For injectivity, the operators \(R_\gamma|_{\hat B}\) form the commutant, while
\(\mathbb R\oplus\su(2)\) is a commuting quaternion copy.
The resulting identification
\(\mathbb H\otimes\mathbb H^{\mathrm{op}}\cong M_4(\mathbb R)\)
gives injectivity on \(\su(2)\otimes\mathbb H\).

Finally, the image is \(\mathbf3\)-isotypic and has dimension \(12\), so it is
all of \(W_{\mathbf3}\) by Lemma~\ref{lem:su2-branching}. Its multiplicity space
is therefore the label space \(\mathbb H\).

By Lemma~\ref{lem:su2-branching},
\(W_{\mathbf4}=(\mathbb H\otimes B)\oplus(\hat B\otimes A)\cong\hat B\otimes\mathbb R^7=\hat B^7\cong\mathbb H^7\),
with \(K\) acting by left multiplication in each slot, and \(\hat B\) is of
quaternionic type with commutant the right quaternion multiplications;
on seven copies these give \(M_7(\mathbb H)\), the operators
\(\svec\mapsto\svec\mathcal M\), \(\mathcal M\in M_7(\mathbb H)\), acting on rows
\(\svec\in\mathbb H^7\) from the right.

  We next show that $\Gamma$ preserves each component.
  Since $K$ is the pointwise stabilizer of
  $\mathcal E=\mathrm{span}_{\mathbb R}\{u,v\}$,
  every $g\in\Gamma$ normalizes $K$.
  Every automorphism of $K\cong\SU(2)$ is inner, so
  there is $k_0\in K$ such that
  \[
    g^{-1}kg=k_0kk_0^{-1}
    \qquad\text{for }k\in K.
  \]
  Hence $gk_0$ commutes with $K$ and preserves its
  isotypic components. Since $k_0$ also preserves
  them, so does $g$.

\end{proof}

These sectors are common invariant subspaces for all operators in the pencil. Their symmetry also
reduces the variables needed for the determinant calculations.

\begin{proposition}[Dependence on the invariants]
  \label{prop:sector-invariants}
  The restriction homomorphism $\Gamma\to O(2)$, $g\mapsto\kappa$,
  is surjective. For each $\mathbf i\in\{\mathbf1,\mathbf3,\mathbf4\}$,
  the sector determinant
  \[
    D_{\mathbf i}(\avec)=\det(P_\avec|_{W_{\mathbf i}})
  \]
  satisfies
  \[
    D_{\mathbf i}(\kappa\cdot\avec)=D_{\mathbf i}(\avec)
    \qquad\text{for }\kappa\in O(2).
  \]
  Consequently, $D_{\mathbf i}$ is a polynomial in $a_0$, $x$ and $y$.
\end{proposition}

\begin{proof}
  Recall that $\Gtwo$ acts transitively
  on ordered orthonormal pairs in $\Im\mathbb O$.
  Consequently, every $\kappa\in O(2)$ is induced by
  some $g\in\Gamma$. By \eqref{eq:frame-conjugation}
  and the $\Gamma$-invariance of each $W_{\mathbf i}$,
  \[
    D_{\mathbf i}(\kappa\cdot\avec)
      =D_{\mathbf i}(\avec),
  \]
  proving the claimed invariance.

  Finally, fix $\mathbf i$ and set
  \[
    f(a_0,\tau,\upsilon)=D_{\mathbf i}(a_0,\tau,0,\upsilon).
  \]
  The action of $\kappa=-I$ makes $f$ even in $\tau$,
  while $\kappa=\operatorname{diag}(-1,1)$ then makes
  it even in $\upsilon$. Thus, $f(a_0,\tau,\upsilon)=F(a_0,\tau^2,\upsilon^2)$
  for some polynomial $F$. Since a rotation can take
  $(a_1,a_2)$ to $(\sqrt{x},0)$ without changing $a_3$,
  rotational invariance gives
  \[
    D_{\mathbf i}(\avec)
      =F(a_0,x,a_3^2)
      =F(a_0,x,y).
  \]
\end{proof}

We call $W_{\mathbf1}$ the fixed sector since it is preserved pointwise by $\SU(2)$. In this sector, the part $\mathbb H\otimes A$ lies in $W_+$ and the
four-dimensional space of the $\varepsilon_\gamma$ lies in $W_-$. A quaternionic
coordinate system shows how the pencil acts on both parts at once.

\begin{proposition}
  \label{prop:w1-normal-form}
  Identify $W_{\mathbf1}$ with $\mathbb H^3\oplus\mathbb H$
  via the real-linear isomorphism
  \begin{equation*}
    \begin{aligned}
      \Phi:W_{\mathbf1}&\longrightarrow\mathbb H^3\oplus\mathbb H,\\
      \sum_{j=1}^3\vartheta_j\otimes\rho_j+\varepsilon_\gamma
        &\longmapsto(\vartheta_1,\vartheta_2,\vartheta_3;\eta),
        \qquad\text{where}\qquad \eta=z\gamma\bar z.
    \end{aligned}
  \end{equation*}
  We label the first three quaternionic coordinates by
  $1,2,3$ and the last by $0$.
  Then, for every $\avec\in\mathbb R^4$,
  the operator $a_0I+a_1R_u+a_2R_v+a_3R_vR_u$
  acts by right multiplication by $\hat q$ on all four slots:
  \begin{equation}
    (\vartheta_1,\vartheta_2,\vartheta_3;\eta)
      \longmapsto
    (\vartheta_1\hat q,\vartheta_2\hat q,\vartheta_3\hat q;\eta\hat q).
    \label{eq:w1-scalar-part}
  \end{equation}
  The matrix of $h$ in these coordinates is given in
  Appendix~\ref{app:fixed-sector-calculation},
  equation~\eqref{eq:fixed-sector-h-blocks}.
\end{proposition}

\begin{proof}
  By Proposition~\ref{prop:sign-twist}, $R_vR_u=R_z$ on $W_+=\mathbb H\otimes V$, hence 
  the operator acts by right multiplication by $\hat q$
  on slots $1,2,3$, which correspond to $\mathbb{H} \otimes A \subset W_+$.

  On slot $0$, the Cayley--Dickson rule
  \eqref{eq:complement-products} gives
  \[
    R_p\varepsilon_\gamma=\varepsilon_{\gamma\bar p}
    \qquad\text{and}\qquad
    R_vR_u\varepsilon_\gamma=\varepsilon_{\gamma z}.
  \]
  Hence the operator acts on $\gamma$ by
  \[
    \gamma\longmapsto
    \gamma(a_0-a_1u-a_2v+a_3z)
    =\gamma(\bar z\hat q z).
  \]
  Under the change of variable $\eta=z\gamma\bar z$,
  this becomes
  \[
    \eta\longmapsto
    z\gamma(\bar z\hat q z)\bar z
    =\eta\hat q.
  \]
  Thus all four slots carry right multiplication
  by $\hat q$, as claimed.
\end{proof}


The tensor metric in these coordinates is
\begin{equation*}
  \|\Phi^{-1}(\boldsymbol\vartheta;\eta)\|^2
   =\sum_{j=1}^3|\vartheta_j|^2+4|\eta|^2,
  \qquad\text{with}\qquad G=\diag(I_{12},4I_4).
\end{equation*}
Indeed $\|\varepsilon_\gamma\|^2=4|\gamma|^2$ and
$|z\gamma\bar z|=|\gamma|$.

These coordinates are not orthonormal: the fourth quaternionic slot has
metric weight $4$. This weight accounts for the transpose relation
$h_{A0}=4h_{0A}^{\mathsf T}$ used in
Appendix~\ref{app:fixed-sector-calculation}.

\section{The two-step determinant and spectral locus}
\label{sec:determinant-locus}

The invariant sectors of Theorem~\ref{thm:sectors} reduce the determinant
calculation to three blocks. Table~\ref{tab:sector-summary} summarizes the
conclusions proved below: the determinant theorems
\ref{thm:fixed-sector-determinant}, \ref{thm:three-block-determinant} and
\ref{thm:four-block-determinant}, the invertibility theorem
\ref{thm:four-block-positivity}, and the real-locus theorem
\ref{thm:real-locus}. The factors $Q_{14},Q_{12}$ and sets $Z,S$ are defined
in the corresponding subsections.

\begin{table}[ht]
\centering
\begin{tabular}{c|c|c|l}
Sector & Dimension & Real determinant & Real singular locus\\\hline
$W_{\mathbf1}$ & $16$ & $NQ_{14}$ & $Z$\\
$W_{\mathbf3}$ & $12$ & $[N(N+4a_0)]^3$ & $S$\\
$W_{\mathbf4}$ & $28$ & $[NQ_{12}]^2$ & contained in $Z\cap\{a_3=0\}$
\end{tabular}
\caption{Sector contributions to the real two-step locus.}
\label{tab:sector-summary}
\end{table}

\subsection{The fixed sector \texorpdfstring{$W_{\mathbf1}$}{W1}}
\label{subsec:fixed-sector-block}

We use the fixed-sector coordinates
$(\vartheta_1,\vartheta_2,\vartheta_3;\eta)\in\mathbb H^3\oplus\mathbb H$
of Proposition~\ref{prop:w1-normal-form}.
The first three slots, indexed by $A=(1,2,3)$, correspond to
$\mathbb H\otimes A\subset W_+$, while slot $0$ corresponds to
$\varepsilon_\gamma\in W_-$. In these coordinates, the scalar
part of $P_\avec$ acts by $R_{\hat q}$ on every slot.

Writing $\Phi h\Phi^{-1}$ in blocks
$h_{AA},h_{A0},h_{0A},h_{00}$ as in
\eqref{eq:fixed-sector-h-blocks}, we obtain
\begin{equation*}
  \Phi P_\avec\Phi^{-1}
  =\begin{pmatrix}
    P_{AA} & P_{A0}\\
    P_{0A} & P_{00}
  \end{pmatrix}
  =\begin{pmatrix}
    h_{AA}-R_{\hat q}^{\oplus3} & h_{A0}\\
    h_{0A} & h_{00}-R_{\hat q}
  \end{pmatrix},
\end{equation*}
where $R_{\hat q}^{\oplus3}$ denotes componentwise right
multiplication on $\mathbb H^3$.

To analyze the $AA$ block, let $\mathbf r=(u,v,z)^{\mathsf T}$, and let
$\mathbf r^*=\bar{\mathbf r}^{\mathsf T}$ denote its
quaternionic conjugate transpose. Then $h_{AA}=L_M$, where
\[
  M=-\mathbf r\mathbf r^*,
  \qquad\text{and}\qquad \mathbf r^*\mathbf r=3.
\]
Thus $\mathbb H^3$ decomposes into right quaternionic
submodules
\[
  \mathbb H^3=U_3\oplus U_0,
  \qquad\text{where}\qquad
  U_3=\mathbf r\mathbb H\quad\text{and}\quad
  U_0=\ker\mathbf r^*,
\]
of ranks one and two, respectively.
On these submodules, $P_{AA}$ acts as
\[
  P_{AA}|_{U_3}=-R_{\hat q+3}
  \qquad\text{and}\qquad
  P_{AA}|_{U_0}=-R_{\hat q}.
\]
Consequently, using $N=|\hat q|^2$ and $N_3=|\hat q+3|^2=N+6a_0+9$, we obtain
\begin{equation*}
  \det P_{AA}=N^4N_3^2.
\end{equation*}
When $NN_3\ne0$, the inverse of $P_{AA}$ is given by
\eqref{eq:fixed-sector-aa-inverse}. The projections
and the corresponding calculations are provided in
Appendix~\ref{app:fixed-sector-calculation}.

For $NN_3\ne0$, we eliminate the $\mathbb H^3$ block
using the Schur complement
\begin{equation*}
  \mathcal S
  =P_{00}-P_{0A}P_{AA}^{-1}P_{A0}.
\end{equation*}
This is a real-linear operator on $\mathbb H$, represented
by a real $4\times4$ matrix in the basis $(1,u,v,z)$.
Lemma~\ref{lem:fixed-sector-polynomial-matrix} constructs the real
$4\times4$ polynomial matrix $M_{\mathrm{pol}}$, displayed in
equation~\eqref{eq:fixed-sector-matrix}.
Its entries belong to
$\mathbb Q[a_0,a_1,a_2,a_3]$, and we define
\begin{equation*}
  Q_{14}:=\det M_{\mathrm{pol}}
  \in\mathbb Q[a_0,a_1,a_2,a_3].
\end{equation*}
This polynomial definition applies at all coefficient points, including
$NN_3=0$. Appendix~\ref{app:fixed-sector-calculation} supplies the
determinant normalization and continuation argument used in the theorem below.

The explicit expression is shorter in the variables $(a_0,N,t)$,
where $t=x-y$. The change from $(a_0,x,y)$ is polynomial with
polynomial inverse over $\mathbb Q$:
\begin{equation*}
  x=\tfrac12(N-a_0^2+t),\qquad\text{and}\qquad y=\tfrac12(N-a_0^2-t).
\end{equation*}

\begin{theorem}[Determinant on the fixed sector]
  \label{thm:fixed-sector-determinant}
  Using notation as above:
  \begin{enumerate}
    \item[\textup{(i)}]
      The determinant of $P_\avec$ on $W_{\mathbf1}$
      is given by
      \begin{equation}
        \det(P_\avec|_{W_{\mathbf1}})=NQ_{14}.
        \label{eq:fixed-sector-determinant}
      \end{equation}
      On $NN_3\ne0$, the Schur complement
      $\mathcal S$ satisfies
      \begin{equation*}
        \begin{aligned}
          Q_{14}
            &=N^3N_3^2\det\nolimits_{\mathbb R}\mathcal S.
        \end{aligned}
      \end{equation*}

    \item[\textup{(ii)}]
      As an element of $\mathbb Q[x,y][a_0]$, the polynomial
      $Q_{14}(a_0,x,y)$ is monic of degree $14$ in $a_0$.
      For its explicit expression, write $N=a_0^2+x+y$ and $t=x-y$.
      Then:
\begin{equation}
\label{eq:fixed-sector-q14-ant}
\begin{aligned}
Q_{14}={}&256\,N(2a_0+3)^2\,t^2+q_1\,t+q_0,  \text{where}\\
q_1={}&N\bigl[96(a_0+1)N^3+32(26a_0^2+37a_0+23)N^2\\
&\quad+128(4a_0^3-36a_0^2-55a_0-16)N-1024a_0^2(2a_0+3)(2a_0+9)\bigr]\\
\text{and}\quad q_0={}&N^7+2(8a_0+9)N^6+3(2a_0+1)(6a_0+11)N^5\\
&-4(108a_0^3+192a_0^2+159a_0+92)N^4\\
&-32a_0(54a_0^3+29a_0^2-52a_0-30)N^3\\
&+256a_0(4a_0^4+74a_0^3+88a_0^2-21a_0-52)N^2\\
&+256a_0^2(7a_0+8)(4a_0^3+12a_0^2-31a_0-28)N\\
&-8192a_0^4(2a_0+3)(2a_0+7).
\end{aligned}
\end{equation}
  \end{enumerate}
\end{theorem}

\begin{proof}
  Proposition~\ref{prop:fixed-sector-q14-determinant}
  gives the Schur normalization and global polynomial factorization in (i),
  including at
  $NN_3=0$. Proposition \ref{prop:sector-invariants} then implies that $Q_{14}$ is a polynomial in
  $(a_0,x,y)$. This proves (i) and the
  polynomial dependence asserted in (ii).
  The appendix evaluates the Schur determinant using
  Lemma~\ref{lem:diagonal-skew-determinant}; see also
  \cite[\S8, (9), pp.~106--107]{cohen-2000-quaternionic}.

  To check monicity, $P_\avec+a_0I$ is independent of $a_0$.
  Since $\dim_{\mathbb R}W_{\mathbf1}=16$,
  $\det(P_\avec|_{W_{\mathbf1}})$ is monic of
  degree $16$ in $a_0$. As
  $N=a_0^2+x+y$ is monic of degree $2$,
  the factorization in (i) shows that $Q_{14}$
  is monic of degree $14$ in $a_0$.

\end{proof}

The alternate expression~\eqref{eq:fixed-sector-q14-closed-form}
gives the same coefficients after $x,y$ are written in terms of $(a_0,N,t)$.
For the direct evaluation, substitute
$N=a_0^2+a_1^2+a_2^2+a_3^2$ and $t=a_1^2+a_2^2-a_3^2$
into~\eqref{eq:fixed-sector-q14-ant}.
Exact symbolic determinant calculation shows that the determinant of the
polynomial matrix~\eqref{eq:fixed-sector-matrix} minus this expression
expands to zero in $\mathbb Q[a_0,a_1,a_2,a_3]$.

\begin{definition}
We write
\begin{equation}
\label{eq:def-Z}
Z=\bigl\{\avec\in\mathbb R^4:Q_{14}(a_0,a_1^2+a_2^2,a_3^2)=0\bigr\}
\end{equation}
for the real zero set of $Q_{14}(a_0,x,y)$ in the coefficient space.
We also write $Q_{14}(\avec)$ for its pullback
$Q_{14}(a_0,a_1^2+a_2^2,a_3^2)$ when $\avec\in\mathbb R^4$.
\end{definition}

\subsection{The \texorpdfstring{$W_{\mathbf3}$}{W3} sector}
\label{subsec:three-block}

By Theorem~\ref{thm:sectors}, the sector $W_{\mathbf3}$ lies in $\hat B\otimes B$, hence entirely in the reversed half $W_-=\hat B\otimes V$. The right-multiplication term on this sector is $R_{\check q}$. First we determine the action of the right product and $h$ on elements of the form $w_{X,\gamma}$, as defined in \eqref{eq:epsilon}.

\begin{lemma}
  \label{lem:three-block-label-actions}
  For $X\in\su(2)$ and $\gamma,p\in\mathbb H$,
  \begin{equation*}
    R_pw_{X,\gamma}=w_{X,p\gamma}
    \qquad\text{and}\qquad
    hw_{X,\gamma}=w_{X,-4\Re\gamma}.
  \end{equation*}
\end{lemma}

\begin{proof}
  Recall from \eqref{eq:epsilon} that
  \[
    w_{X,\gamma}
      =\sum_{\alpha=0}^3
        ((X\hat\xi_\alpha)\gamma)\otimes\xi_\alpha.
  \]
  Since $X$ preserves $\hat B$, the Cayley--Dickson
  identity $(\hat\xi\gamma)p=\hat\xi(p\gamma)$
  from \eqref{eq:complement-products} gives
  \begin{align*}
    R_pw_{X,\gamma}
      &=\sum_{\alpha=0}^3
        (((X\hat\xi_\alpha)\gamma)p)\otimes\xi_\alpha\\
      &=\sum_{\alpha=0}^3
        ((X\hat\xi_\alpha)(p\gamma))\otimes\xi_\alpha =w_{X,p\gamma}.
  \end{align*}

  To compute the action of $h$, identify
  $\hat B\otimes B$ with
  $\operatorname{End}_{\mathbb R}(\hat B)$
  using the paired orthonormal bases
  $(\hat\xi_\alpha)$ and $(\xi_\alpha)$:
  a tensor $a\otimes\xi_\alpha$ corresponds to
  the map $\hat\xi_\beta\mapsto\delta_{\alpha\beta}a$.
  Thus $w_{X,\gamma}$ corresponds to the endomorphism
  \[
    T_{X,\gamma}=R_\gamma X
    \qquad \text{with} \qquad
    T_{X,\gamma}(\hat\xi_\alpha)
      =(X\hat\xi_\alpha)\gamma.
  \]
  Under the fixed identification $\hat V\cong V$,
  this endomorphism, extended by zero on $A$,
  represents the corresponding tensor in $V\otimes V$. By real linearity, it suffices to consider $\gamma=1$ and $\gamma\in\Im\mathbb H$.

  For $\gamma=1$, we have $T_{X,1}=X$.
  Since $X\in\su(2)\subset\mathfrak g_2$ acts
  as a skew-adjoint endomorphism of $V$,
  its corresponding alternating $2$-tensor lies in
  $\Lambda^2_{14}V$. Hence, using the real spectrum of $h$ from Proposition \ref{prop:real-spectrum},
  \[
    hw_{X,1}=-4w_{X,1}.
  \]

  Now suppose $\gamma\in\Im\mathbb H$.
  Both $X$ and $R_\gamma$ are skew-adjoint
  on $\hat B$. Since $X$ is a derivation
  vanishing on $\mathbb H$,
  \[
    X(\hat\xi\gamma)=(X\hat\xi)\gamma,
  \]
  so $X$ commutes with $R_\gamma$.
  Consequently, $T_{X,\gamma}=R_\gamma X$
  is self-adjoint and therefore corresponds
  to a symmetric $2$-tensor. This tensor is also trace-free. Indeed, for
  fixed $\gamma$, the functional
  \[
    X\longmapsto\operatorname{tr}(R_\gamma X)
  \]
  is $\operatorname{Ad}(K)$-invariant, since
  $R_\gamma$ commutes with $K$. As $\su(2)$
  has no nonzero invariant linear functionals,
  the trace vanishes. Thus
  $w_{X,\gamma}\in S^2_0V$, where $h$ acts
  by zero (Section~\ref{sec:one-step-spectrum}). 
  
  Combining the two cases gives
  \[
    hw_{X,\gamma}=w_{X,-4\Re\gamma}
  \]
  for every $\gamma\in\mathbb H$.

\end{proof}

\begin{proposition}
  \label{prop:three-block-reduction}
  \leavevmode
  \begin{enumerate}
    \item[\textup{(i)}]
     On \(W_{\mathbf{3}}\), the two-step pencil agrees with the restriction of the one-step pencil \(h-R_{\check q}\). For a pair with reduced point \(\mathbf{a}\), this is the one-step comparison operator at \(\check q=\lambda_2\lambda_1\).

    \item[\textup{(ii)}]
      Under the $K$-equivariant identification
      $\su(2)\otimes\mathbb H\cong W_{\mathbf3}$
      given by $X\otimes\gamma\mapsto w_{X,\gamma}$
      in \eqref{eq:epsilon},
      \begin{equation*}
        P_\avec|_{W_{\mathbf3}}
          =I_{\su(2)}\otimes P_{\mathbf3},
        \qquad \text{where} \qquad
        P_{\mathbf3}(\gamma)
          =-4\Re\gamma-\check q\gamma.
      \end{equation*}
  \end{enumerate}
\end{proposition}

\begin{proof}
  Since $\hat B\otimes V\subset W_-$,
  part~(i) follows from Proposition~\ref{prop:halves},
  with $\check q=\lambda_2\lambda_1$ for a reduced
  pair by Proposition~\ref{prop:general-reduction}.
  The explicit action in part~(ii) follows from
  Lemma~\ref{lem:three-block-label-actions} above.
\end{proof}

We can now compute the determinant of the pencil on the sector $W_{\mathbf3}$.

\begin{theorem}
  \label{thm:three-block-determinant}
  Recall from Proposition~\ref{prop:three-block-reduction}
  that $P_\avec|_{W_{\mathbf3}}$ is identified with
  $I_{\su(2)}\otimes P_{\mathbf3}$, where
  $P_{\mathbf3}:\mathbb H\to\mathbb H$ is the
  real-linear map $P_{\mathbf3}(\gamma) =-4\Re\gamma-\check q\gamma.$ Then:
  \begin{enumerate}
    \item[\textup{(i)}]
      The determinant satisfies the polynomial identity
      \begin{equation}
        \det\nolimits_{\mathbb R} (P_\avec|_{W_{\mathbf3}})
          =p_{\mathbf3}^3,
        \qquad \text{where} \qquad
        p_{\mathbf3}
          =N(N+4a_0)
          =|\check q|^2\bigl(|\check q+2|^2-4\bigr).
        \label{eq:three-block-determinant}
      \end{equation}
      In particular, $p_{\mathbf3}$ depends on $\avec$
      only through $a_0$ and $x+y$.

    \item[\textup{(ii)}]
      For $\avec \in\mathbb R^4$, the restriction
      $P_\avec|_{W_{\mathbf3}}$ is singular precisely
      on the round $3$-sphere
      \begin{equation*}
        S=\left\{\avec\in\mathbb R^4:
          (a_0+2)^2+a_1^2+a_2^2+a_3^2=4
        \right\}.
      \end{equation*}
      This sphere contains the origin and lies in
      $-4\le a_0\le0$.

    \item[\textup{(iii)}]
      The kernel of $P_{\mathbf3}$ is
      \begin{equation*}
        \ker P_{\mathbf3}=
        \begin{cases}
          \{0\}, & \avec\notin S,\\
          \mathbb R\overline{\check q},
            & \avec\in S\setminus\{0\},\\
          \Im\mathbb H, & \avec=0.
        \end{cases}
      \end{equation*}
      Under the identification
      $W_{\mathbf3}\cong\su(2)\otimes\mathbb H$,
      the kernel of $P_\avec|_{W_{\mathbf3}}$
      corresponds to
      \[
        \su(2)\otimes\ker P_{\mathbf3},
      \]
      and therefore has real dimension $0$, $3$,
      or $9$, respectively.
  \end{enumerate}
\end{theorem}

\begin{proof}
   In the real basis $(1,u,v,z)$ of $\mathbb H$,
  let $e_0=(1,0,0,0)^{\mathsf T}$.
  Since $e_0e_0^{\mathsf T}$ represents the projection
  $\gamma\mapsto\Re\gamma$, we have
  \[
    P_{\mathbf3}=-(L_{\check q}+4e_0e_0^{\mathsf T}).
  \]
  For $N=|\check q|^2\ne0$, left multiplication
  by $\check q$ satisfies
  \[
    \det\nolimits_{\mathbb R} L_{\check q}=N^2
    \qquad\text{and}\qquad
    L_{\check q}^{-1}=\frac1N L_{\overline{\check q}}.
  \]
  In particular,
  $e_0^{\mathsf T}L_{\check q}^{-1}e_0=a_0/N$.
  We can thus compute the determinant as:
  \begin{align*}
    \det\nolimits_{\mathbb R} P_{\mathbf3}
      &=N^2\left(1+4e_0^{\mathsf T}
          L_{\check q}^{-1}e_0\right)\\
      &=N^2\left(1+\frac{4a_0}{N}\right)=N(N+4a_0).
  \end{align*}
  Both sides are polynomials in $\avec$, so the
  identity extends to $N=0$. Since
  $P_\avec|_{W_{\mathbf3}}=I_{\su(2)}\otimes P_{\mathbf3}$
  and $\dim_{\mathbb R}\su(2)=3$, its determinant
  is $p_{\mathbf3}^3$, proving (i).

  For real $\avec\in\mathbb R^4$, $N=0$ only at the origin.
  Moreover,
  \[
    N+4a_0
      =(a_0+2)^2+a_1^2+a_2^2+a_3^2-4.
  \]
  Thus the zero set of $p_{\mathbf3}$ is exactly
  the sphere $S$, proving (ii). Finally, the kernel equation is
  \begin{equation*}
    \check q\gamma=-4\Re\gamma.
  \end{equation*}
  If $\check q\ne0$, multiplication by
  $\check q^{-1}=\overline{\check q}/N$ gives
  \[
    \gamma=-\frac{4\Re\gamma}{N}\,
      \overline{\check q}.
  \]
  Hence every kernel element is a real multiple
  of $\overline{\check q}$. Conversely,
  \[
    P_{\mathbf3}(\overline{\check q})
      =-(N+4a_0),
  \]
  so this line is the kernel precisely when
  $\avec\in S\setminus\{0\}$.

  At the origin, $\check q=0$, and
  $P_{\mathbf3}(\gamma)=-4\Re\gamma$, whose kernel
  is $\Im\mathbb H$. The corresponding kernels
  on $W_{\mathbf3}$ follow by tensoring with
  $\su(2)$, giving the stated dimensions.
\end{proof}

Thus $S$ is the one-step circle condition $|\check q+2|=2$
at the reversed product (Corollary~\ref{cor:general-pairs}).

\subsection{The \texorpdfstring{$W_{\mathbf4}$}{W4} sector}
\label{subsec:four-block}

The sector $W_{\mathbf4}=(\mathbb H\otimes B)\oplus(\hat B\otimes A)$
is the direct sum of seven copies of the four-dimensional real
$K$-module of quaternionic type. We first choose equivariant coordinates
in which the multiplication part of $P_\avec$ acts by the same quaternion
$\hat q$ on all seven slots. We then identify the resulting quaternionic
matrix and compute its reduced norm. Finally, a sum of squares determines
its real zero set.

In the original halves, the multiplication term involves $\hat q$ on
$W_+$ and $\check q$ on $W_-$. The coordinates below conjugate the action
on the second part to the common row action $\hat q I_7$; the coupling
between the halves remains in the off-diagonal blocks of $M_h$.

For $\gamma\in\mathbb H$ and $\hat\xi\in\hat B$, define
\[
  \Psi_\gamma(\hat\xi)
    =\tfrac12\sum_{i=0}^{3}\hat\rho_i\otimes
      \bigl((\overline{\hat\rho_i}\gamma)\hat\xi\bigr),
\]
where the second factor is in $B\subset V$. For
$f=(f_0,\dots,f_6)\in\hat B^7$, define
\begin{equation}
  \mathcal J(f)=\sum_{k=0}^{3}\Psi_{\hat\rho_k}(f_k)
    +\sum_{j=1}^{3}(f_{3+j}z)\otimes\rho_j.
  \label{eq:four-sector-coordinates}
\end{equation}
This is an isometric, $K$-equivariant isomorphism; the proof below also
shows that $\Psi_\gamma$ is independent of the orthonormal basis of
$\mathbb H$ used in its sum.

Fix a unit $\ell\in\hat B$ and write $f_k=s_k\ell$ with
$\svec=(s_0,\dots,s_6)\in\mathbb H^7$.
For an equivariant operator $T$, its matrix $\mathcal M_T$ is specified by
$T\mathcal J(f)=\mathcal J(f')$, where
$f'_k=\sum_{i=0}^{6}(\mathcal M_T)_{ik}f_i$.
A quaternion entry acts on $\hat B$ by left multiplication. Since
$\gamma(s\ell)=(s\gamma)\ell$, the induced coordinate action is
$T_{\mathcal M_T}(\svec)=\svec\mathcal M_T$.
Thus $T_{\mathcal N}\circ T_{\mathcal M}=T_{\mathcal M\mathcal N}$:
$uI_7$ followed by $vI_7$ gives $zI_7$.

In these coordinates the full matrix of $h$ is
\begin{equation}
  \begin{gathered}
    M_h=\begin{pmatrix}M_{++}&M_{+-}\\ M_{-+}&M_{--}\end{pmatrix},\\
    M_{++}=\diag(2,-2,-2,-2)\qquad\text{and}\qquad M_{-+}=M_{+-}^*,\\
    M_{+-}=\begin{pmatrix}
      -v&u&-1\\
      z&-1&-u\\
      1&z&-v\\
      u&v&z
    \end{pmatrix}\qquad\text{and}\qquad
    M_{--}=\begin{pmatrix}
      -1&-z&-v\\
      z&-1&u\\
      v&-u&-1
    \end{pmatrix}
  \end{gathered}
  \label{eq:four-sector-matrix}
\end{equation}
where $M_{+-}^*$ is the conjugate transpose. Equivalently, its indexed
entries are $(M_{+-})_{k,3+j}=\hat\rho_j\overline{\hat\rho_k}z$
and $(M_{--})_{3+i,3+j}=\bar z\hat\rho_j\hat\rho_i z$, for
$k=0,\dots,3$ and $i,j=1,2,3$. The first four slots correspond to $\mathbb H\otimes B\subset W_+$ and
the last three to $\hat B\otimes A\subset W_-$, whence the subscripts.

For an equivariant operator $T$ represented by
$\mathcal M\in M_7(\mathbb H)$, its real determinant is
\begin{equation*}
  \det\nolimits_{\mathbb R}T=\Nrd(\mathcal M)^2
  \qquad\text{where}\qquad
  \Nrd(\mathcal M)=\det\nolimits_{\mathbb C}\chi(\mathcal M),
\end{equation*}
where $\Nrd$ denotes the quaternionic reduced norm
and $\chi:M_7(\mathbb H)\to M_{14}(\mathbb C)$
is induced by a unital real-algebra embedding
$\mathbb H\hookrightarrow M_2(\mathbb C)$.
The reduced norm is real, nonnegative, and independent
of the choice of embedding. It is the Study determinant
$\operatorname{Sdet}$ of Aslaksen, and
$\det\nolimits_{\mathbb R}=\operatorname{Sdet}^2$
\cite[Theorems~5 and~9, and the remark on p.~63]{aslaksen-1996-quaternionic}.
For definition of $\chi$, its multiplicativity, and the nonnegativity of $\det\nolimits_{\mathbb C}\chi$,
see \cite[Theorems~4.2 and~8.1, Proposition~4.2 and its corollary]{zhang-1997-quaternions}.

We use the explicit embedding
\begin{equation}
  \chi(\alpha+\beta u+\gamma v+\varepsilon z)=
  \begin{pmatrix}
    \alpha+i\beta&\gamma+i\varepsilon\\
    -\gamma+i\varepsilon&\alpha-i\beta
  \end{pmatrix},
  \qquad\text{for }\alpha,\beta,\gamma,\varepsilon\in\mathbb R,
  \label{eq:app-c-complex-embedding}
\end{equation}
and apply $\chi$ entrywise to quaternionic matrices.

\begin{proposition}
  \label{prop:four-block-form}
  \leavevmode
  \begin{enumerate}
    \item[\textup{(i)}]
      Under $\mathcal J$, the algebra of $K$-equivariant operators on
      $W_{\mathbf4}$ is $M_7(\mathbb H)$, acting on quaternionic row vectors
      from the right.
    \item[\textup{(ii)}]
      The operators $R_u$, $R_v$ and $R_vR_u$ are represented by
      $uI_7$, $vI_7$ and $zI_7$, respectively, and $h$ by the
      quaternion-Hermitian matrix~\eqref{eq:four-sector-matrix}.
      Consequently $P_\avec|_{W_{\mathbf4}}$ is represented by
      $M_h-\hat qI_7$, and
      \[
        \det\nolimits_{\mathbb R}(P_\avec|_{W_{\mathbf4}})
          =\Nrd(M_h-\hat qI_7)^2.
      \]
  \end{enumerate}
\end{proposition}

\begin{proof}
  \emph{Determinant convention.}
  The equivariant identification with $\hat B^7$ gives the commutant
  $M_7(\mathbb H)$ by Theorem~\ref{thm:sectors}(d), with the row action
  just described. This proves (i).

  For a matrix $\mathcal M$, the row map
  $T_{\mathcal M}:\svec\mapsto\svec\mathcal M$ commutes with left multiplication by
  $\mathbb C=\mathbb R\oplus\mathbb Ru$, so it is $\mathbb C$-linear on
  $\mathbb H^7\cong\mathbb C^{14}$. Its complex matrix is the transpose of the
  entrywise image $\chi'(\mathcal M)$ under a unital $\mathbb R$-algebra embedding
  $\chi'\colon\mathbb H\to M_2(\mathbb C)$, so its complex determinant is
  $\det\chi'(\mathcal M)$.

  Any two unital $\mathbb R$-algebra embeddings
  $\mathbb H\to M_2(\mathbb C)$ are conjugate, by the Skolem--Noether theorem for
  $\mathbb H\otimes\mathbb C\cong M_2(\mathbb C)$; hence
  $\det\chi'(\mathcal M)=\det\chi(\mathcal M)$, and the choice of $\chi$ does not
  matter.

  This determinant is real because $\chi(\mathcal M)$ commutes with a
  conjugate-linear map $J$ with $J^2=-1$. Finally, the real determinant
  of a $\mathbb C$-linear map is the squared modulus of its complex determinant,
  so $\det\nolimits_{\mathbb R}T_{\mathcal M}=\Nrd(\mathcal M)^2$, since $\Nrd(\mathcal M)$ is real.

  \emph{Temporary coordinates.}
  To prove (ii), we first compute in tensor-product coordinates and then
  make a unitary change that sends the multiplication part to right
  multiplication by $\hat q$ in every slot. The same change gives
  the displayed matrix of $h$ and the map $\mathcal J$ in
  \eqref{eq:four-sector-coordinates}.
  Fix a unit $\ell\in\hat B$, so that $\mathbb O=\mathbb H\oplus\mathbb H\ell$ with
  the Cayley--Dickson rules~\eqref{eq:cayley-dickson-rules}.
  The temporary coordinates $\boldsymbol\omega=(\omega_0,\dots,\omega_6)\in\mathbb H^7$ are given by
  \[
    w=\sum_{k=0}^{3}\hat\rho_k\otimes(\omega_k\ell)
      +\sum_{j=1}^{3}(\omega_{3+j}\ell)\otimes\rho_j,
  \]
  where the second factor of the first sum is read in $B\subset V$; they are not
  the coordinates of $\mathcal J$. The group $\SU(2)$ acts by
  $\omega_k\mapsto g\omega_k$ (Appendix~\ref{app:su2-bases}, equation~\eqref{eq:frame-stabilizer-action}), so equivariant operators act on these rows from the right.

  \emph{Matrix before the coordinate change.}
  By Definition~\ref{def:operator-h} and bilinearity,
  $h(p\otimes b)=\sum_i((\hat e_i\hat b)p)\otimes e_i$ for $p\in\mathbb O$ and
  $b\in V$, where $\hat b=\sum_jb_j\hat e_j$ for $b=\sum_jb_je_j$
  (equation~\eqref{eq:h-definition}). Expanding this with the Cayley--Dickson rules
  gives the following matrix entries in the temporary coordinates, for
  $i,k=0,\dots,3$ and $j,l=1,2,3$:
  \[
  \begin{aligned}
    (M'_h)_{ik}&=-\hat\rho_k\overline{\hat\rho_i},
    & (M'_h)_{i,3+j}&=\hat\rho_j\overline{\hat\rho_i},\\
    (M'_h)_{3+j,k}&=-\hat\rho_k\hat\rho_j,
    & (M'_h)_{3+j,3+l}&=\hat\rho_l\hat\rho_j.
  \end{aligned}
  \]

  For the scalar part, $R_p$ with
  $p\in\mathbb H$ acts on the first four slots by the real matrix with entries
  $\langle\hat\rho_k,\hat\rho_ip\rangle$ ($i$ the row) and on the last
  three by $\omega\mapsto\omega\bar p$. By the sign-twist identity
  $R_vR_u=R_z\circ(\sigma\otimes1)$ of Proposition~\ref{prop:sign-twist},
  $R_vR_u$ is $R_z$ on the first
  four slots and $\omega\mapsto\omega z$ on the last three.

  \emph{Common scalar action.}
  Change coordinates by
  \[
  \begin{aligned}
    \eta_k&=\tfrac12\sum_{i=0}^{3}\omega_i\,\overline{\hat\rho_k}\,\hat\rho_i,
      && k=0,\dots,3,\\
    \eta_{3+j}&=\omega_{3+j}z, && j=1,2,3.
  \end{aligned}
  \]
  Equivalently,
  \[
    \boldsymbol\eta=\boldsymbol\omega\mathcal T,
    \qquad\text{where}\qquad
    \mathcal T=\diag(T,zI_3),\quad
    T_{ik}=\tfrac12\overline{\hat\rho_k}\hat\rho_i.
  \]
  The matrix $T\in M_4(\mathbb H)$ is unitary.
  The raw coordinates are isometric, and the unitary change commutes with
  the left $K$-action. Hence the resulting map $\mathcal J$ is an
  equivariant isometric isomorphism. On the first four slots this
  change intertwines $R_p$ with $\eta\mapsto\eta p$.

  On the last three slots, the action
  $\omega\mapsto\omega\bar p$ becomes $\eta\mapsto\eta\bar z\bar pz$, which is $\eta p$ for
  $p=u,v$, and the action $\omega\mapsto\omega z$ of $R_vR_u$ becomes $\eta\mapsto\eta z$.
  So the scalar part $a_0+a_1R_u+a_2R_v+a_3R_vR_u$ becomes $\eta\mapsto\eta\hat q$
  on all seven slots.

  \emph{Transformed matrix.}
  Conjugating the blocks of $M'_h$ by $\mathcal T$, with the identities
  $\sum_{k=0}^{3}\overline{\hat\rho_k}\gamma\overline{\hat\rho_k}=-2\bar\gamma$ and
  $\sum_{k=0}^{3}\overline{\hat\rho_k}\gamma\hat\rho_k=4\Re\gamma$ for
  $\gamma\in\mathbb H$ (Appendix~\ref{app:su2-bases}, equation~\eqref{eq:quaternion-sums}), gives the upper-left block with entries
  $2\Re(\overline{\hat\rho_i}\,\overline{\hat\rho_k})$, which is
  $\diag(2,-2,-2,-2)$, the coupling $\hat\rho_j\overline{\hat\rho_k}z$, and the
  lower-right block $\bar z\,\hat\rho_j\hat\rho_i\,z$ entrywise. This is
  \eqref{eq:four-sector-matrix}. The matrix $M_h$ is Hermitian because $h$ is symmetric and
  the coordinates are isometric.

  \emph{Independence of choices.}
  Finally, undoing the coordinates with $f_k=\eta_k\ell$ gives
  \eqref{eq:four-sector-coordinates}, in which neither $\ell$ nor a basis of $\hat B$ appears.
  If $r'_i=\sum_j O_{ij}r_j$ is another orthonormal basis of $\mathbb H$,
  with $O$ real orthogonal, substitution into the defining sum for
  $\Psi_\gamma$ contracts the two occurrences of $O$ by
  $\sum_iO_{ij}O_{ik}=\delta_{jk}$. Thus that sum is basis independent.
\end{proof}

The chosen coordinates make the scalar part equal to $\hat qI_7$
on all seven slots, while the off-diagonal blocks
$M_{+-}$ and $M_{-+}$ of $M_h$ couple $W_+$ and $W_-$.
The row-vector convention is essential: quaternionic
transposition does not generally preserve the reduced norm,
and in particular
$\Nrd((M_h-\hat qI_7)^{\mathsf T})$
can differ from $\Nrd(M_h-\hat qI_7)$.
See also \cite{cohen-2000-quaternionic}.
The sources cited above let matrices act on column vectors, whereas
$\mathcal M$ acts here on row vectors from the right. Conjugating the
coordinates, $\svec\mapsto\bar{\svec}$, turns the row action of $\mathcal M$
into the column action of the conjugate transpose $\mathcal M^{*}$, and
$\operatorname{Sdet}\mathcal M^{*}=\operatorname{Sdet}\mathcal M$
\cite[pp.~61--62]{aslaksen-1996-quaternionic}, so the identities above apply
to $\mathcal M$ as written.

We now evaluate the reduced norm of the displayed matrix $M_h-\hat qI_7$. Let $p_{\mathbf4}=\Nrd(M_h-\hat qI_7)$, so that
$\det(P_\avec|_{W_{\mathbf4}})=p_{\mathbf4}^2$ by
Proposition~\ref{prop:four-block-form}. Recall that $t=x-y$.

\begin{theorem}
\label{thm:four-block-determinant}
  As an identity in $\mathbb Q[a_0,a_1,a_2,a_3]$, after substituting
  $N=a_0^2+a_1^2+a_2^2+a_3^2$ and $t=a_1^2+a_2^2-a_3^2$, we have $p_{\mathbf4}=N\,Q_{12}$, where the polynomial
  $Q_{12}$ is given by
  \begin{align}
    Q_{12}&=F_0(a_0,N)+F_1(a_0,N)\,t+64N\,t^2, \qquad \text{where}
    \label{eq:app-d-q12}\\
    F_1&=16\bigl[(3a_0+3)N^3+(28a_0^2+40a_0+8)N^2\notag\\
    &\qquad\quad+(64a_0^3+104a_0^2+64a_0)N-64a_0^3\bigr],\notag\\
    F_0&=N^6+(14a_0+17)N^5+(32a_0^2+144a_0+64)N^4\notag\\
    &\quad+(-304a_0^3+160a_0^2+512a_0+64)N^3\notag\\
    &\quad+(-1600a_0^4-256a_0^3+3200a_0^2+1024a_0)N^2\notag\\
    &\quad+(-2048a_0^5+3136a_0^4+12288a_0^3+4096a_0^2)N\notag\\
    &\quad+7168a_0^5+12288a_0^4.\notag
  \end{align}
  In particular, $p_{\mathbf4}$ and $Q_{12}$ are polynomials in
  $(a_0,N,t)$, hence in $(a_0,x,y)$, with rational coefficients,
  and $Q_{12}$ has degree $2$ in $t$.
\end{theorem}

\begin{proof}
  Apply the embedding~\eqref{eq:app-c-complex-embedding} to the
  quaternionic matrix $M_h-\hat qI_7$ specified in~\eqref{eq:four-sector-matrix}.
  Exact symbolic expansion of this $14\times14$ complex determinant gives
  \begin{equation*}
    \det\nolimits_{\mathbb C}\chi(M_h-\hat qI_7)
      =N\bigl(F_0(a_0,N)+F_1(a_0,N)t+64Nt^2\bigr).
  \end{equation*}
  By definition its left-hand side is $p_{\mathbf4}$, proving the identity.
\end{proof}

Since $NQ_{12}=p_\mathbf{4}=\det\nolimits_{\mathbb C}\chi(M_h-\hat qI_7)\geq 0$, we find immediately that for $N\ne0$, $Q_{12}\geq0$. The following sum of squares proves strict positivity of $Q_{12}$ when
$y>0$, and hence invertibility of the block whenever $a_3\ne0$, since $y=a_3^2$.

\begin{theorem}[Positivity and invertibility of the quaternionic sector]
  \label{thm:four-block-positivity}
  Set $\delta=N+4a_0$ and
  \begin{align*}
    G_0&=\delta\bigl[N^2+(3a_0+4)N-16a_0(a_0+2)\bigr]-8y(N+6a_0),\\
    G_1&=\delta(3N+16a_0+16)-16y,\\
    G_2&=\delta(5N+8a_0-16)+16y.
  \end{align*}
  Then, as a polynomial identity in $\mathbb Q[a_0,x,y]$,
  \begin{equation}
    Q_{12}=G_0^2+xG_1^2+yG_2^2+320xy\delta^2.
    \label{eq:app-c-q12-sos}
  \end{equation}
  Consequently, $Q_{12}>0$ for real $a_0$, $x\ge0$ and $y>0$.

  In particular, for real $\avec$ with $a_3\ne0$,
  $P_\avec|_{W_{\mathbf4}}$ is invertible.
\end{theorem}

\begin{proof}
Substituting $N=a_0^2+x+y$, $t=x-y$ into~\eqref{eq:app-d-q12},
  exact polynomial expansion gives~\eqref{eq:app-c-q12-sos}.
  Now suppose $y>0$. If $x>0$ and $\delta\ne0$, the last term is positive.
  If $x>0$ and $\delta=0$, then $G_1=-16y$, so $xG_1^2>0$.
  It remains to consider $x=0$, where $N=a_0^2+y$ and
  $Q_{12}=G_0^2+yG_2^2$.
  A zero would therefore force $G_0=G_2=0$. Set $L_0=N\delta-16a_0$.
  In this case the polynomial identities
  \begin{align*}
    G_0+(a_0+4)G_2&=(N+8a_0+16)L_0=((a_0+4)^2+y)L_0,\\
    G_2&=5L_0+8a_0(N+2a_0+2)=5L_0+8a_0((a_0+1)^2+y+1)
  \end{align*}
  follow by expansion. Thus the first identity forces $L_0=0$ and the second then forces $a_0=0$. But $a_0=x=0$ gives $N=\delta=y$, so $L_0=y^2>0$, a contradiction.
  This proves strict positivity in all cases.

For real $\avec$ with $a_3\ne0$, we have $x=a_1^2+a_2^2\ge0$,
  $y=a_3^2>0$ and $N=a_0^2+x+y\ge y>0$.
  The strict positivity just proved gives $Q_{12}>0$.
  Thus $p_{\mathbf4}=NQ_{12}>0$, and
  $\det(P_{\avec}|_{W_{\mathbf4}})=(NQ_{12})^2>0$
  by Theorem~\ref{thm:four-block-determinant}.
\end{proof}

The invertibility proof uses $a_3\neq0$. On the hyperplane $a_3=0$ we instead
use the additional symmetry in the following subsection.

\subsection{The one-step hyperplane \texorpdfstring{$a_3=0$}{a3=0}}
\label{subsec:face-a3-zero}

On the hyperplane $a_3=0$, the associator term vanishes,
and every operator \(P_{\avec}\) is a one-step operator, with enhanced symmetry.

Let $Q=\sqrt{x}=\sqrt{a_1^2+a_2^2}$. If $Q\neq0$, set
\[
  \hat n=\frac{a_1u+a_2v}{Q},
\]
and if $Q=0$, take $\hat n=u$. Let $n\in A$ be the
corresponding unit vector in $V$. By linearity of right
multiplication, for $\avec=(a_0,a_1,a_2,0)$ we have
\[
  P_\avec=h-QR_{\hat n}-a_0I=H_{n,Q}-a_0I,
\]
the level-1 slice pencil of
Section~\ref{sec:one-step-spectrum}.
In particular, $P_\avec$ has the symmetry of the
stabilizer $\SU(3)_n$, which contains
$K=\SU(2)=\Stab_{\Gtwo}(u,v)$.

\begin{lemma}[One-step kernel lemma]
  \label{lem:face-lemma}
  Let $\avec=(a_0,a_1,a_2,0)\in\mathbb R^4$.
  Then $P_\avec$ commutes with $\SU(3)_n$,
  and every nonzero $\SU(3)_n$-invariant subspace
  $U\subset W$ contains a nonzero $K$-fixed vector.
  Consequently,
  \begin{equation*}
    \ker P_\avec\ne0
    \quad\Longrightarrow\quad
    \ker P_\avec\cap W_{\mathbf1}\ne0.
  \end{equation*}
\end{lemma}

\begin{proof}
  By Proposition~\ref{prop:equivariance},
  $h$ is $\Gtwo$-equivariant and $R_{\hat n}$
  commutes with $\SU(3)_n$. Hence so does
  $P_\avec=H_{n,Q}-a_0I$.
  Moreover, $K\subset\SU(3)_n$, since $K$
  fixes $\mathbb H$, and therefore $n$, pointwise.

  Under $\SU(3)_n$, the real representation $W$
  decomposes as
  \[
    W\cong
    \mathbf1^{\oplus4}\oplus
    \mathbf6^{\oplus4}\oplus
    \mathbf8^{\oplus2}\oplus
    \mathbf{12}
  \]
  (Section~\ref{sec:one-step-spectrum}).
  These four representation types are irreducible
  over $\mathbb R$ (Section~\ref{sec:one-step-spectrum}).
  Under the standard inclusion
  $K\subset\SU(3)_n$, each contains a nonzero
  $K$-fixed vector
  (Lemma~\ref{lem:su3-to-su2-branching}
  of Appendix~\ref{app:su2-bases}).

  Now let $U\subset W$ be a nonzero
  $\SU(3)_n$-invariant subspace.
  By complete reducibility for the compact group
  $\SU(3)_n$ \cite[Theorem~4.28]{hall-2015-lie},
  $U$ contains a nonzero
  irreducible $\SU(3)_n$-submodule.
  This submodule must be of one of the four types
  above, and hence contains a nonzero $K$-fixed
  vector. Therefore
  \[
    U\cap W_{\mathbf1}=U^K\ne0.
  \]

  Finally, since $P_\avec$ commutes with
  $\SU(3)_n$, its kernel is $\SU(3)_n$-invariant.
  Applying the preceding conclusion to
  $U=\ker P_\avec$ proves the lemma.
\end{proof}

\subsection{The full determinant and real singular locus}
\label{subsec:determinant}
\label{subsec:real-locus}

Multiplying the three sector determinants gives the full determinant below.
Corollary~\ref{cor:general-pairs} expresses it in terms of invariants of a pair.

\begin{theorem}[The two-step determinant]
\label{thm:two-step-determinant}
For the pencil of an orthonormal pair $u,v\in\Im\mathbb O$,
\begin{equation}
  \det\nolimits_{56}P_{\avec}=N^6\,Q_{14}(a_0,x,y)\,(N+4a_0)^3\,Q_{12}(a_0,x,y)^2,
  \label{eq:two-step-determinant}
\end{equation}
where $Q_{14}$ and $Q_{12}$ are the polynomials of Theorems~\ref{thm:fixed-sector-determinant} and~\ref{thm:four-block-determinant}.
\end{theorem}

We now determine the real singular locus of $P_\avec$.
Recall the zero set $Z$ of $Q_{14}$ from \eqref{eq:def-Z} and the sphere $S$
from Theorem~\ref{thm:three-block-determinant}.

\begin{theorem}[The real locus]
  \label{thm:real-locus}
  Let $\Sigma\subset\mathbb R^4$ be the singular locus
  of $P_\avec$, and for
  $\mathbf i\in\{\mathbf1,\mathbf3,\mathbf4\}$,
  let $\Sigma_{\mathbf i}$ be the singular locus
  of its restriction to $W_{\mathbf i}$. Then:
  \begin{enumerate}
    \item[\textup{(i)}]
      The sector singular loci satisfy
      \[
        \Sigma_{\mathbf1}=Z,\qquad
        \Sigma_{\mathbf3}=S\qquad\text{and}\qquad
        \Sigma_{\mathbf4}\subset Z\cap\{a_3=0\}.
      \]

    \item[\textup{(ii)}]
      The full real singular locus is
      \[
        \Sigma=Z\cup S.
      \]
  \end{enumerate}
\end{theorem}

\begin{proof}
  The invariant sector decomposition gives
  \[
    \ker P_\avec
      =\bigoplus_{\mathbf i}
        \ker(P_\avec|_{W_{\mathbf i}}),
    \qquad\text{and}\qquad
    \Sigma=\bigcup_{\mathbf i}\Sigma_{\mathbf i}.
  \]

  By Theorem~\ref{thm:fixed-sector-determinant},
  the determinant on $W_{\mathbf1}$ is $NQ_{14}$.
  Over $\mathbb R$, $N$ vanishes only at the origin,
  which also belongs to $Z$ since $M_{\mathrm{pol}}=0$
  there (Lemma~\ref{lem:fixed-sector-polynomial-matrix}).
  Thus $\Sigma_{\mathbf1}=Z$.
  Theorem~\ref{thm:three-block-determinant} gives
  $\Sigma_{\mathbf3}=S$.

  For the remaining sector,
  Theorem~\ref{thm:four-block-positivity}
  shows that
  \[
    \Sigma_{\mathbf4}\subset\{a_3=0\}.
  \]
  On this hyperplane, Lemma~\ref{lem:face-lemma} implies that
  $P_\avec$ is singular if and only if its
  restriction to $W_{\mathbf1}$ is singular.
  Hence
  \[
    \Sigma\cap\{a_3=0\}
      =Z\cap\{a_3=0\}.
  \]
  In particular, $\Sigma_{\mathbf4}\subset Z$,
  completing (i). Finally, the sector decomposition
  gives
  \[
    \Sigma
      =\Sigma_{\mathbf1}\cup\Sigma_{\mathbf3}
         \cup\Sigma_{\mathbf4}
      =Z\cup S,
  \]
  proving (ii).
\end{proof}

\begin{remark}
\label{rem:one-step-specialization}
Define $N_0=a_0^2+x$ and $C_0=N_0+4a_0$, and write
$\widetilde p_{\mathbf1}(a_0,x)=p_{\mathbf1}(a_0,\sqrt{x})$.
This is a polynomial because the quartic~\eqref{eq:p1} is even in $Q$.
For $y=0$, we find that
\begin{equation*}
 Q_{14}(a_0,x,0)=C_0\widetilde p_{\mathbf1} E_8\qquad\text{and}\qquad Q_{12}(a_0,x,0)=C_0^2E_8,
\end{equation*}
Here
\[
 E_8(a_0,x)=\left|p_{\mathbf6}(a_0,\sqrt{x})\right|^2,
\]
where $p_{\mathbf6}$ is the complex quartic determinant of the one-step
$\mathbf6^{\oplus4}$ multiplicity block
\cite{grigorian-2026-octonionic}. The singlet quartic $p_{\mathbf1}$ of
\eqref{eq:p1} is the separate factor $\widetilde p_{\mathbf1}$ above.
By Proposition~\ref{prop:level1-spectrum}, $E_8>0$ for $x>0$, so it adds
no nonreal one-step spectral curve.

\end{remark}

At every point of $Z\cap S$, the kernels in
$W_{\mathbf1}$ and $W_{\mathbf3}$ are both nontrivial
and contribute independently to $\ker P_\avec$.

\renewcommand{\topfraction}{0.9}
\renewcommand{\floatpagefraction}{0.8}
Define the invariant-image map
\[
  \pi:\mathbb R^4\longrightarrow\mathbb R\times[0,\infty)^2,
  \qquad\text{by}\qquad
  \pi(\avec)=(a_0,a_1^2+a_2^2,a_3^2).
\]
Figure~\ref{fig:real-locus} plots $\pi(Z)$ and $\pi(S)$ in
invariant-image coordinates. The sphere $S$ is round in $\mathbb R^4$, but its invariant image is
$$
\pi(S)=\{(a_0,x,y):x,y\ge0,\ x+y=4-(a_0+2)^2\}.
$$
Thus its image is a portion of a parabolic cylinder.

The boundary $y=0$ is the invariant image of the coefficient hyperplane
$a_3=0$. On this hyperplane, $P_{\avec}=H_{n,Q}-a_0$
(Section~\ref{subsec:face-a3-zero}). Thus the trace on $y=0$ is the image
of the level-1 slice locus under $(\mu,Q)\mapsto(a_0,x)=(\mu,Q^2)$.
Its real-axis trace is the real spectrum of $h$ given in
Proposition~\ref{prop:real-spectrum}.

\begin{figure}[t]
  \centering
  \includegraphics{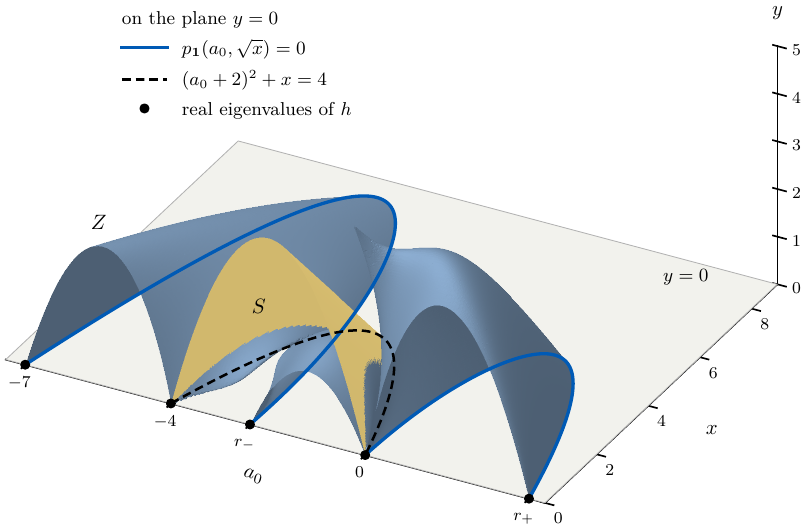}
  \caption{Invariant images of the real locus in $(a_0,x,y)$, with $x,y\ge0$: blue-gray is the image of $Z$ and pale yellow the image of $S$. On $y=0$, the marks of Figure~\ref{fig:level1-slice-locus} are mapped by $(\mu,Q)\mapsto(a_0,x)=(\mu,Q^2)$: blue is the quartic image, dashed is the one-step circle image, and dots are the five real eigenvalues. The dashed curve is the trace of the image of $S$ on the boundary plane $y=0$. }
  \label{fig:real-locus}
\end{figure}

The following corollary expresses the determinant and real-locus results
in invariants of an arbitrary octonion pair. Recall that $a_0=\Re(\lambda_1\lambda_2)$, $N=\lvert\lambda_1\rvert^2\lvert\lambda_2\rvert^2$, that $y=a_3^2$ is the Gram determinant of $\Im\lambda_1$ and $\Im\lambda_2$, and that $x=N-a_0^2-y$ \eqref{eq:invariants}.
\begin{corollary}
  \label{cor:general-pairs}
  Let $(\lambda_1,\lambda_2)$ be any pair of octonions and let $\avec$ be its reduced point. 
  \begin{enumerate}
    \item The determinant $\det\nolimits_{56}(h-R_{\lambda_2}R_{\lambda_1})$ is given by \eqref{eq:two-step-determinant} at $(a_0,x,y)$, with $N=a_0^2+x+y=\lvert\lambda_1\rvert^2\lvert\lambda_2\rvert^2$.
    \item The equation $h(w)=(w\lambda_1)\lambda_2$ has a nonzero solution $w\in W$ if and only if $\avec\in Z\cup S$, that is, if and only if $Q_{14}(a_0,x,y)=0$ or
    \begin{equation}
      \lvert\lambda_2\lambda_1+2\rvert=2
      \quad\Longleftrightarrow\quad
      \lvert\lambda_1\rvert^2\lvert\lambda_2\rvert^2+4\Re(\lambda_1\lambda_2)=0.
      \label{eq:general-pair-sphere}
    \end{equation}
    \item If $\Im\lambda_1$ and $\Im\lambda_2$ are linearly independent, that is, $a_3\neq0$, the $W_{\mathbf4}$ block of $h-R_{\lambda_2}R_{\lambda_1}$ is invertible.
  \end{enumerate}
\end{corollary}

The following corollary measures how far a two-step kernel vector is
from satisfying any one-step equation. It also identifies exactly
which kernel vectors have no octonionic right eigenvalue.

\begin{corollary}
  \label{cor:neither-product}
  Let $\avec\in\mathbb R^4$ and
  $0\ne w\in\ker P_\avec$, with sector decomposition
  $w=w_1+w_3+w_4$. Set
  $\alpha=\lVert\Pi_+w\rVert^2$ and
  $\beta=\lVert\Pi_-w\rVert^2$. Then
  \begin{equation}
    \min_{\lambda\in\mathbb O}\lVert h(w)-w\lambda\rVert^2
      =\frac{4a_3^2\alpha\beta}{\alpha+\beta}.
    \label{eq:right-eigenvalue-residual}
  \end{equation}
  The unique minimizing octonion is
  \[
    \lambda_*=
    \frac{\alpha\hat q+\beta\check q}{\alpha+\beta}.
  \]
  Consequently,
  \begin{equation*}
    h(w)\ne w\lambda\ \text{for every }\lambda\in\mathbb O
    \quad\Longleftrightarrow\quad
    a_3\ne0\ \text{and}\ w_1\ne0.
  \end{equation*}
  This is equivalent to satisfying neither of the equations
  $h(w)=w\hat q$ and $h(w)=w\check q$.
  Such a kernel vector exists precisely when $\avec\in Z$ and $a_3\ne0$.
\end{corollary}

\begin{proof}
  If $a_3=0$, then $P_\avec=h-R_{\hat q}$ and
  $\hat q=\check q$. Every kernel vector satisfies the common
  one-step equation, and both sides of
  \eqref{eq:right-eigenvalue-residual} vanish.

  Suppose $a_3\ne0$. Theorem~\ref{thm:four-block-positivity}
  gives $w_4=0$. By Theorem~\ref{thm:sectors},
  \[
    W_{\mathbf1}=(\mathbb H\otimes A)\oplus(\hat B\otimes B)^K,
    \qquad W_{\mathbf3}\subset\hat B\otimes B.
  \]
  Thus $\Pi_+w$ has its vector index in $A$, and
  $\Pi_-w$ has its vector index in $B$.
  Proposition~\ref{prop:halves} gives, for every $\lambda\in\mathbb O$,
  \[
    h(w)-w\lambda
      =(\Pi_+w)(\hat q-\lambda)+(\Pi_-w)(\check q-\lambda).
  \]
  The two terms are orthogonal because $A\perp B$ in the
  second tensor factor; this does not require right multiplication
  by $\lambda$ to preserve $\mathbb H$ or $\hat B$.
  Multiplicativity of the octonion norm therefore yields
  \begin{equation}
    \lVert h(w)-w\lambda\rVert^2
      =\alpha|\hat q-\lambda|^2+\beta|\check q-\lambda|^2.
    \label{eq:weighted-right-residual}
  \end{equation}
  Since $\alpha+\beta=\lVert w\rVert^2>0$, this quadratic is
  strictly convex. Completing the square gives the stated minimizer
  and minimum, using $|\hat q-\check q|^2=4a_3^2$.

  It remains to characterize $\alpha\beta>0$.
  If $w_1=0$, then $w=w_3\in W_-$ and $\alpha=0$.
  If $w_1\ne0$, write its fixed-sector coordinates as
  $(\boldsymbol\vartheta;\eta)\in\mathbb H^3\oplus\mathbb H$.
  Since $N,N_3>0$, the block $P_{AA}$ is invertible, so
  $\eta=0$ forces $\boldsymbol\vartheta=0$.
  Conversely, $\boldsymbol\vartheta=0$ forces $h_{A0}\eta=0$,
  and the injective block $h_{A0}$ in
  \eqref{eq:fixed-sector-h-blocks} then gives $\eta=0$.
  Hence both projections of $w_1$ are nonzero. Moreover,
  \[
    \Pi_+w=\Pi_+w_1
    \qquad\text{and}\qquad
    \Pi_-w=\Pi_-w_1+w_3,
  \]
  where the last two terms cannot cancel: the projections are
  $K$-equivariant and these terms belong to distinct isotypic sectors.
  Thus $\alpha\beta>0$ exactly when $w_1\ne0$.
  The minimum is attained, so its positivity proves the criterion.
  Evaluating \eqref{eq:weighted-right-residual} at $\hat q$ and
  $\check q$ gives $4a_3^2\beta$ and $4a_3^2\alpha$, respectively,
  proving the neither-product equivalence.
  Finally, Theorem~\ref{thm:real-locus} gives a nonzero fixed-sector
  kernel exactly on $Z$, proving the existence assertion.
\end{proof}

Corollary~\ref{cor:neither-product} concerns an individual kernel vector.
The following example also excludes both products from the entire
one-step right spectrum, although a nonzero two-step solution exists.

\begin{example}[A genuinely two-step spectral point]
  \label{rem:new-spectrum-point}
  Let $b=\sqrt6/4$ and $c=\sqrt{22}/4$. The octonion pair
  \[
    \lambda_1=cu\qquad\text{and}\qquad
    \lambda_2=\sqrt{\frac3{11}}+\frac6{\sqrt{22}}u+v
  \]
  has a one-dimensional two-step kernel, while both
  $\lambda_1\lambda_2$ and $\lambda_2\lambda_1$ lie outside the
  one-step right spectrum. Its imaginary parts are linearly independent
  and nonorthogonal, and its reduced point is
  $\avec_*=(-3/2,b,0,c)$.
  Here $N=N_3=4$, $x=3/8$ and $y=11/8$.
\end{example}

\begin{proof}
  The fixed-sector Schur complement, in the basis
  $(1,u,v,z)$, is
  \[
    \mathcal S_*=
    \begin{pmatrix}
       2&0&0&2c\\
       0&-2&0&0\\
       0&0&-2&-2b\\
       -2c&0&2b&-2
    \end{pmatrix}.
  \]
  Since $c^2-b^2=1$, its kernel contains
  $(-c,0,-b,1)^{\mathsf T}$. The upper-left
  $3\times3$ minor is $8$, so
  \[
    \ker\mathcal S_*
      =\mathbb R(-c,0,-b,1)^{\mathsf T}.
  \]
  Thus $Q_{14}(\avec_*)=0$, and the fixed-sector
  kernel is one-dimensional. Moreover,
  $N+4a_0=-2$ and $a_3=c\ne0$, so the
  $\mathbf3$- and $\mathbf4$-sectors are invertible.
  Consequently, $\dim_{\mathbb R}\ker P_{\avec_*}=1.$

  The associated quaternions are
  \[
    \hat q=-\frac32+bu+cz,
    \qquad\text{and}\qquad
    \check q=-\frac32+bu-cz.
  \]
  Both lie outside the level-1 spectrum.
  Indeed, their imaginary parts have norm
  $\sqrt7/2$, and the one-step quartic and
  circle conditions give, respectively,
  \[
    p_{\mathbf1}\left(-\frac32,\frac{\sqrt7}{2}\right)
      =32\ne0
    \qquad\text{and}\qquad
    \left(-\frac32+2\right)^2+\frac74
      =2\ne4.
  \]
  Thus both one-step operators \(h-R_{\widehat q}\) and \(h-R_{\check q}\) are invertible,
  whereas $P_{\avec_*}$ has a nontrivial kernel.
\end{proof}

\section{Conclusion}
\label{sec:conclusion}

The level-0 real spectrum of $h$ is finite, while its level-1 right spectrum is described
in each complex slice by a quartic and a circle
(Section~\ref{sec:one-step-spectrum}). At level~2, the determinant of the two-step
pencil is given by Theorem~\ref{thm:two-step-determinant}, and its real singular
locus is exactly $Z\cup S$ (Theorem~\ref{thm:real-locus}). These results apply to
every pair at its reduced point (Corollary~\ref{cor:general-pairs}).

Corollary~\ref{cor:neither-product} gives the exact minimum of
$\lVert h(w)-w\lambda\rVert^2$ over all octonions $\lambda$.
Away from $a_3=0$, a kernel vector has no octonionic right eigenvalue
exactly when its fixed-sector component is nonzero, and such vectors
exist exactly on $Z$. Thus failure of both natural product equations
already excludes every one-step right eigenvalue for that vector.
Example~\ref{rem:new-spectrum-point} gives a one-dimensional kernel at a
pair whose two products both lie outside the one-step spectrum.


The same questions can be asked of longer sequences of successive right
multiplications, with the order of the factors retained. These sequences
enlarge the family of comparison operators while reducing their generic
symmetry with $h$. For unit factors, let
\[
\mathcal N_k
=\{R_{q_k}\cdots R_{q_1}:|q_i|=1\}
\subset\mathrm{SO}(8),
\qquad\text{and}\qquad \mathcal N_0=\{\pm I\}.
\]
Conway and Smith give the dimensions $n_k=\dim\mathcal N_k$
listed in Table~\ref{tab:higher-level-dimensions}
\cite[Chapter~8, \S8.5, Theorem~9, pp.~94--95]
{conway-2003-quaternions}.
Since $h$ is $\Gtwo$-equivariant, the natural reduced ambient
space for the level-$k$ problem is
\[
\mathcal M_k
=\mathbb R_{>0}\times\bigl(\mathcal N_k/\operatorname{Ad}(\Gtwo)\bigr),
\]
where $\Gtwo$ acts by conjugation on $\mathcal N_k$, leaving the positive
scale factor separate. At level zero, the zero operator is included separately. If $K_k$ denotes the generic
stabilizer of a composite operator, the normalized quotient
has generic dimension
$d_k=n_k-14+\dim K_k$, and the overall positive scale
gives $\dim\mathcal M_k=D_k=d_k+1$ on its principal stratum.

At level 1 a generic operator recovers an
imaginary direction, with stabilizer $\SU(3)$; at level two its vector and
bivector parts recover a plane pointwise, with stabilizer $\SU(2)$.
For levels three and beyond, the key fact is that a triple of octonions that does not lie in a quaternionic subalgebra generates all of \(\mathbb O\), so any octonion automorphism fixing it is the identity. It can then be easily shown (but is outside the scope of the current paper), that generic stabilizers of the corresponding composite operators will also be the identity for $k\geq3$.

\begin{table}[ht]
\centering
\begin{tabular}{c c c c c}
\hline
$k$ & $n_k$ & Generic $K_k$ & $d_k$ & $D_k$ \\
\hline
0 & 0  & $\Gtwo$  & 0 & 1 \\
1 & 7  & $\SU(3)$ & 1 & 2 \\
2 & 13 & $\SU(2)$ & 2 & 3 \\
3 & 18 & $\{e\}$  & 4 & 5 \\
4 & 22 & $\{e\}$   & 8 & 9 \\
\hline
\end{tabular}
\caption{Dimensions of the families of successive right multiplications
and their generic $G_2$-reduced parameter spaces at levels $0$--$4$. The stabilizers and reduced dimensions refer to generic configurations; special configurations may have larger stabilizers.}
\label{tab:higher-level-dimensions}
\end{table}

The spectral locus is a subset of this ambient space, defined
by the vanishing of the $\Gtwo$-invariant determinant
\[
\Delta(\rho A)=\det(h-\rho A\otimes I_V),
\qquad\text{for }\rho>0\text{ and }A\in\mathcal N_k.
\]
Thus the level-one quartic and circle, and the level-two
loci defined by $Q_{14}=0$ and the sphere equation, describe
spectral subsets of reduced ambient spaces of dimensions
two and three, respectively. At three steps, generic residual
symmetry disappears, but $\mathcal N_3$ still has codimension
ten in $\mathrm{SO}(8)$. The higher-level problems therefore
remain constrained spectral problems even without a
nontrivial generic stabilizer. At seven steps,
$\mathcal N_7=\mathrm{SO}(8)$
\cite[Chapter~8, \S8.4, Theorem~7, p.~93]{conway-2003-quaternions}, so all levels can be viewed as
restrictions of the same determinant to nested families in
the full reduced space
$(\mathbb R_{>0} \times \mathrm{SO}(8))/\operatorname{Ad}(\Gtwo)$.

\appendix
\section{Frame, Cayley--Dickson rules and branching}
\label{app:su2-bases}

We use the same notation as in Section~\ref{subsec:frame}.
Multiplication on $\mathbb{H}$ is given by
\begin{equation}
\label{eq:frame-products}
  \hat\rho_i\hat\rho_j
    =-\delta_{ij}+\sum\nolimits_{k=1}^{3}\epsilon_{ijk}\hat\rho_k,
  \qquad\text{for }1\leq i,j\leq3.
\end{equation}

For every unit \(\ell\in\hat B\), the Cayley--Dickson doubling
\cite[Section~2.2, equations~(2) and~(3)]{BaezOcto}
(see also the right-coordinate formula in
\cite[Chapter~6, Appendix~1, p.~80]{conway-2003-quaternions}) gives
\(\mathbb O=\mathbb H\oplus\mathbb H\ell\),
\(\hat B=\mathbb H\ell\) and gives the following rules for any $p,r,s\in\mathbb H$:
\begin{equation}
\label{eq:cayley-dickson-rules}
\begin{aligned}
  p(r\ell)&=(rp)\ell,\\
  (r\ell)p&=(r\bar p)\ell,\\
  \text{and}\quad (p\ell)(s\ell)&=-\bar s p.
\end{aligned}
\end{equation}
These rules hold for every unit \(\ell\perp\mathbb H\).
In particular, taking a unit $\hat\xi\in\hat B$, we obtain:
\begin{equation}
\label{eq:complement-products}
\begin{aligned}
  r\hat\xi&=\hat\xi\bar r,\\
  (\hat\xi r)p&=\hat\xi(p r),\\
  \text{and}\quad(r \hat\xi)p&=(r \bar p)\hat\xi.
\end{aligned}
\end{equation}
The quaternion sums used in the block calculations are
\begin{equation}
\label{eq:quaternion-sums}
\begin{aligned}
  \sum\nolimits_{j=1}^{3}\hat\rho_jp\hat\rho_j
    &=-p-2\bar p,
  &\text{and}\quad\sum\nolimits_{k=0}^{3}\hat\rho_k r \overline{\hat\rho_k}
    &=4\Re r,\\
  \sum\nolimits_{k=0}^{3}\overline{\hat\rho_k} r \overline{\hat\rho_k}
    &=-2\bar r,
  &\text{and}\quad\sum\nolimits_{k=0}^{3}\overline{\hat\rho_k}r\hat\rho_k
    &=4\Re r.
\end{aligned}
\end{equation}

The group \(K\) fixes \(\mathbb H\) pointwise (it fixes \(u\) and \(v\)) and
acts by \cite[Chapter~6, Appendix~2, p.~81]{conway-2003-quaternions}
\begin{equation}
\label{eq:frame-stabilizer-action}
  p+r\ell\longmapsto p+(gr)\ell,\qquad\text{for }g\in\Sp(1)\cong K.
\end{equation}
Let \(\mathbf1,\mathbf3,\mathbf4\) denote the irreducible real \(K\)-modules
of dimensions \(1,3,4\) (the \(K\)-types).
Here \(\mathbf3\) is the real adjoint module, and \(\mathbf4\) is the
underlying real defining module \(\mathbb C^2\), of quaternionic type
\cite[Section~2.2, pp.~5--6, Lemma~2.3]{mashimo-2019-nonflat};
see also \cite[Section~26.3, Proposition~26.24]{fulton-1991-representation}.
In these coordinates \(\hat B\cong\mathbb H\), with left quaternion
multiplication, is the type \(\mathbf4\), has no fixed
vector, and satisfies
\(\hat B_{\mathbb C}\cong\mathbb C^2\oplus\mathbb C^2\), two copies of the
defining (spin-\(\tfrac12\)) module \(\mathbb C^2\) of \(\SU(2)\).
Its commutant is the four-dimensional algebra of right quaternion
multiplications in these coordinates, equivalently the octonionic maps
\(\hat\xi\mapsto\gamma\hat\xi\), by
\eqref{eq:complement-products}.

\begin{lemma}
\label{lem:su2-branching}
Let \(W=\mathbb O\otimes V\), \(\mathbb O=\mathbb H\oplus\hat B\) and
\(V=A\oplus B\).
As real \(K\)-modules, \(\mathbb H\otimes A\) is trivial of dimension \(12\),
and
\begin{equation*}
  (\mathbb H\otimes B)\oplus(\hat B\otimes A)
    \cong\mathbf4^{\oplus7}\qquad\text{and}\qquad
  \hat B\otimes B\cong\mathbf1^{\oplus4}\oplus\mathbf3^{\oplus4}.
\end{equation*}
Consequently \(\dim W^K=16\), and the \(K\)-isotypic sectors of
Section~\ref{subsec:sectors} satisfy
\(W_{\mathbf3}\subset\hat B\otimes B\), with \(\dim W_{\mathbf3}=12\), and
\(W_{\mathbf4}=(\mathbb H\otimes B)\oplus(\hat B\otimes A)\), with
\(\dim W_{\mathbf4}=28\).
\end{lemma}
\begin{proof}
The group \(K\) is trivial on \(\mathbb H\) and on \(A\), so
\(\mathbb H\otimes A\cong\mathbf1^{\oplus12}\).
Since \(B\cong\hat B\), the summands \(\mathbb H\otimes B\) and
\(\hat B\otimes A\) give four and three copies of \(\mathbf4\), respectively.
For the remaining tensor product,
\((\hat B\otimes B)_{\mathbb C}\cong4(\mathbb C^2\otimes\mathbb C^2)
\cong4(\mathbf1\oplus\mathbf3)\), by the Clebsch--Gordan decomposition
\(\mathbb C^2\otimes\mathbb C^2\cong S^2\mathbb C^2\oplus\Lambda^2\mathbb C^2\)
into irreducible modules of dimensions \(3\) and \(1\)
\cite[p.~428 and Theorem~C.1]{hall-2015-lie}.
Thus \(\hat B\otimes B\) gives four trivial and four adjoint copies.
Since \(\mathbb H\otimes A\) has only type \(\mathbf1\),
\((\mathbb H\otimes B)\oplus(\hat B\otimes A)\) only type \(\mathbf4\) and
\(\hat B\otimes B\) only types \(\mathbf1,\mathbf3\), the isotypic sectors are as
stated, and \(\dim W^K=12+4\).
\end{proof}

Recall the real $\SU(3)_n$-modules of Section~\ref{sec:one-step-spectrum},
whose real irreducibility is standard
\cite[Section~2, pp.~58--60]{moroianu-2008-deformations}.
The group $\SU(3)_n$ acts on $\mathbf8=\su(n^\perp,J_n)$ by the adjoint
action, so that the complexification of $\mathbf8$ is
$\mathfrak{sl}_3(\mathbb C)$. These descriptions and the branching rules
are those of
\cite[Definition~5.6, Propositions~5.8 and~5.11]{grigorian-2026-octonionic}.

\begin{lemma}
\label{lem:su3-to-su2-branching}
Let \(n\in A\subset V\) be a unit, the copy of a unit \(\hat n\in\hat A\), and
\(\SU(3)_n=\Stab_{\Gtwo}(n)\). Then \(K\subset\SU(3)_n\) is the standard block
embedding, and the real modules \(\mathbf1,\mathbf6,\mathbf8,\mathbf{12}\) of
Section~\ref{sec:one-step-spectrum} restrict as
\begin{equation}
\label{eq:su3-to-su2-branching}
\begin{aligned}
  \mathbf1\downarrow_K&=\mathbf1,
  &\text{and}\quad\mathbf6\downarrow_K&=2\cdot\mathbf1\oplus\mathbf4,\\
  \mathbf8\downarrow_K&=\mathbf1\oplus\mathbf3\oplus\mathbf4,
  &\text{and}\quad\mathbf{12}\downarrow_K
    &=2\cdot\mathbf1\oplus2\cdot\mathbf3\oplus\mathbf4,
\end{aligned}
\end{equation}
where the types on the left are \(\SU(3)_n\)-types and those on the right
are \(K\)-types. Consequently each of the four modules has a nonzero
\(K\)-fixed vector.
\end{lemma}
\begin{proof}
The group \(K\) fixes \(\mathbb H\), hence \(A\), pointwise, so it fixes \(n\)
and lies in \(\SU(3)_n\).
It fixes pointwise the complex line \(A\cap n^\perp\) of
\((n^\perp,J_n)\cong\mathbb C^3\), which is \(J_n\)-stable because
\(\hat n\times\hat A\subset\hat A\), and it acts on the complementary
\(B\cong\mathbb C^2\) as \(\SU(2)\).
This is the standard block embedding, under which
\(\mathbb C^3\cong\mathbb C\oplus\mathbb C^2\) as a complex \(K\)-module,
with \(K\) trivial on \(\mathbb C\).
The branchings follow from this and the Clebsch--Gordan decomposition
\cite[p.~428 and Theorem~C.1]{hall-2015-lie};
for the adjoint module compare \cite[Section~6.9, Exercise~4]{hall-2015-lie}.
\end{proof}

The sector calculations produce diagonal-plus-skew real \(4\times4\) matrices.
\begin{lemma}
\label{lem:diagonal-skew-determinant}
For real \(d_0,d_1\) and \(\mathbf w,\mathbf k\in\mathbb R^3\), let
\(\Lambda=\diag(d_0,d_1,d_1,d_1)\) and
\(\Theta=\bigl(\begin{smallmatrix}0&-\mathbf w^{\mathsf T}\\
\mathbf w&[\mathbf k]_\times\end{smallmatrix}\bigr)\), where
\([\mathbf k]_\times\boldsymbol\eta=\mathbf k\times\boldsymbol\eta\), a skew
\(4\times4\) matrix. Then
\begin{equation*}
  \det(\Lambda+\Theta)
    =d_0d_1^3+d_1^2|\mathbf w|^2+d_0d_1|\mathbf k|^2+(\mathbf w\cdot\mathbf k)^2.
\end{equation*}
\end{lemma}
\begin{proof}
For a diagonal matrix \(\Lambda\) and a skew matrix \(\Theta\), the standard
Pfaffian expansion is \(\det(\Lambda+\Theta)=\sum\nolimits_{\mathcal I}
\operatorname{Pf}(\Theta_{\mathcal I})^2\prod\nolimits_{j\notin\mathcal I}\Lambda_{jj}\),
the sum over \(\mathcal I\subset\{0,1,2,3\}\) with \(|\mathcal I|\) even
(principal submatrices \(\Theta_{\mathcal I}\), empty Pfaffian \(1\)).
For \(\Lambda=\diag(d_0,d_1,d_1,d_1)\), the empty index set gives
\(d_0d_1^3\), the full Pfaffian is \(\mathbf w\cdot\mathbf k\), and the
two-index terms give \(d_1^2|\mathbf w|^2+d_0d_1|\mathbf k|^2\).
\end{proof}

\section{The \texorpdfstring{$W_{\mathbf1}$}{W1} calculation and \texorpdfstring{$Q_{14}$}{Q14}}
\label{app:fixed-sector-calculation}

This appendix supplies the fixed-sector blocks, Schur reduction,
polynomial matrix and invariant-polynomial argument used in
Section~\ref{subsec:fixed-sector-block}.

We use the quaternionic normal form of Proposition~\ref{prop:w1-normal-form},
\(\Phi:W_{\mathbf1}\longrightarrow\mathbb H^3\oplus\mathbb H\),
with \(\Phi(\vartheta\otimes\rho_j)=\vartheta\) in slot \(j=1,2,3\).
The slots \(1,2,3\) represent \(\mathbb H\otimes A\subset W_+\),
and slot \(0\) represents \(\{\varepsilon_\gamma:\gamma\in\mathbb H\}\subset W_-\),
with \(\Phi(\varepsilon_\gamma)=z\gamma\bar z\).

Write \(h_{AA},h_{A0},h_{0A},h_{00}\) for the blocks of
\(\Phi h\Phi^{-1}\), and use the same subscripts for
\(\Phi P_{\avec}\Phi^{-1}\). The scalar part is \(R_{\hat q}\) in every slot,
by \eqref{eq:w1-scalar-part}, so that \(P_{AA}=h_{AA}-R_{\hat q}\),
\(P_{A0}=h_{A0}\), \(P_{0A}=h_{0A}\) and \(P_{00}=h_{00}-R_{\hat q}\).
The coordinate metric $G=\diag(I_{12},4I_4)$ gives
$h_{A0}=4h_{0A}^{\mathsf T}$ in Euclidean coordinate matrices.

Use $x,y,N$ from~\eqref{eq:invariants}, $\hat q,\check q$ from
\eqref{eq:normal-reversed}, and $t=x-y$ from
Section~\ref{subsec:fixed-sector-block}. For the calculations below, define
\begin{equation*}
  q_\perp=a_1u+a_2v\qquad\text{and}\qquad N_3=|\hat q+3|^2=N+6a_0+9.
\end{equation*}
We also fix the real vectors \(\mathbf p=\Im\hat q=(a_1,a_2,a_3)\) and
\(\mathbf m=\Im\check q=(a_1,a_2,-a_3)\) in the basis \((u,v,z)\), oriented by
\(u\times v=z\), so that
\(|\mathbf p|^2=|\mathbf m|^2=x+y\) and \(\mathbf p\cdot\mathbf m=x-y=t\).
For a real 3-vector \(\mathbf v\), \([\mathbf v]_\times\) denotes the matrix of
\(\boldsymbol\eta\mapsto\mathbf v\times\boldsymbol\eta\).

\subsection{The normal form and the \texorpdfstring{$\mathbb H^3$}{H3} block}
\label{subsec:fixed-sector-h-block}

\begin{lemma}
\label{lem:fixed-sector-h}
For \((\vartheta_1,\vartheta_2,\vartheta_3;\eta)\in\mathbb H^3\oplus\mathbb H\),
with \(\boldsymbol\vartheta=(\vartheta_1,\vartheta_2,\vartheta_3)\),
the four blocks of \(\Phi h\Phi^{-1}\) are
\begin{equation}
\label{eq:fixed-sector-h-blocks}
\begin{aligned}
 (h_{AA}\boldsymbol\vartheta)_i&=\sum_{j=1}^{3}(\hat\rho_i\hat\rho_j)\vartheta_j,
 &\text{and}\quad (h_{A0}\eta)_i&=-4\bar z\bar\eta z\hat\rho_i,\\
 h_{0A}\boldsymbol\vartheta&=z\Bigl(\sum_{j=1}^{3}\overline{\vartheta_j\hat\rho_j}\Bigr)\bar z,
 &\text{and}\quad h_{00}\eta&=-4\Re\eta.
\end{aligned}
\end{equation}
\end{lemma}

\begin{proof}
We evaluate the defining formula~\eqref{eq:h-definition} for \(h\), over the
orthonormal basis \(\{\rho_1,\rho_2,\rho_3,\xi_0,\ldots,\xi_3\}\) of \(V\).
Take a unit \(\ell\in\hat B\) and \(\hat\xi_\alpha=\hat\rho_\alpha\ell\).
The Cayley--Dickson rules~\eqref{eq:cayley-dickson-rules} of
Appendix~\ref{app:su2-bases} give
\((\ell a)p=\overline{pa}\,\ell\) and
\((a\ell)( r\ell)=-\bar r a\) for \(a,p, r\in\mathbb H\). For a generic element \(\vartheta\ell\) of
\(\hat B\), with \(\vartheta\in\mathbb H\), the contribution
from \(\hat B\otimes B\) is computed before summing:
\begin{equation*}
 ((\vartheta\ell)\hat\xi_\alpha)( r\hat\xi_\alpha)
   =-\bigl((\hat\rho_\alpha r\overline{\hat\rho_\alpha})\vartheta\bigr)\ell,
 \qquad\text{and}\qquad
 \sum_{k=0}^{3}\hat\rho_k r\overline{\hat\rho_k}=4\Re r.
\end{equation*}
Thus the sum over \(\alpha\) is \(-4(\Re r)\,\vartheta\ell\).
For the other blocks we use the first sum of~\eqref{eq:quaternion-sums} and
the rule \(\hat\rho_i\hat\rho_j=-\delta_{ij}+\epsilon_{ijk}\hat\rho_k\).
Summing the four contributions, with \(\eta=z r\bar z\), gives
\eqref{eq:fixed-sector-h-blocks}.
\end{proof}

Let $M_{ij}=\hat\rho_i\hat\rho_j$, so $h_{AA}=L_M$. Set $\mathbf r=(u,v,z)^{\mathsf T}$; then $M=-\mathbf r\mathbf r^*$ and $\mathbf r^*\mathbf r=3$.
Thus $\mathbb H^3=U_3\oplus U_0$, where $U_3=\mathbf r\mathbb H$ and $U_0=\ker\mathbf r^*$ orthogonally, with
projections $\mathbf r\mathbf r^*/3$ and $I-\mathbf r\mathbf r^*/3$ and quaternionic ranks one and two.
On these right submodules $L_M$ is $-3$ and $0$, respectively, so
$P_{AA}$ is $-R_{\hat q+3}$ and $-R_{\hat q}$.
Right multiplication by a quaternion $p$ on $\mathbb H$ has real
determinant $|p|^4$, giving the determinant and inverse below.
It follows that
\begin{equation}
\label{eq:fixed-sector-aa-inverse}
\begin{aligned}
 \det P_{AA}&=N^4N_3^2,\\
 P_{AA}^{-1}&=-\Pi_{U_0}R_{\hat q^{-1}}
                -\Pi_{U_3}R_{(3+\hat q)^{-1}},\qquad\text{for }NN_3\ne0,\\
 \Pi_{U_0}&=I-\mathbf r\mathbf r^*/3,\qquad\text{and}\qquad \Pi_{U_3}=\mathbf r\mathbf r^*/3.
\end{aligned}
\end{equation}

\subsection{Schur reduction and the closed form}

\begin{lemma}
\label{lem:fixed-sector-schur}
Suppose \(NN_3\ne0\). Define
\(\zeta_1=(z\hat q\bar z)^{-1}\) and
\(\zeta_2=(3+z\hat q\bar z)^{-1}\). The Schur complement
\(\mathcal S=P_{00}-P_{0A}P_{AA}^{-1}P_{A0}\) is given by

\begin{equation}
    \mathcal S\eta =\mathcal A\eta+\eta\mathcal B+\mathcal C\bar\eta+\bar\eta\mathcal D,
\end{equation}
where
\begin{equation}
\label{eq:fixed-sector-schur-coefficients}
\begin{aligned}
 \mathcal A&=-2+\tfrac43(2\bar\zeta_1+6\zeta_1+\bar\zeta_2),
 &\quad\mathcal B&=-\hat q-\tfrac{16}{3}(\zeta_1-\zeta_2),\\
 \mathcal C&=-2-\tfrac83(\bar\zeta_1-\bar\zeta_2),
 &\quad\mathcal D&=-\tfrac83(\zeta_1-\zeta_2).
\end{aligned}
\end{equation}
Moreover, \(\det(P_{\avec}|_{W_{\mathbf1}})=N^4N_3^2\det\nolimits_{\mathbb R}\mathcal S\).
\end{lemma}

\begin{proof}
By direct computation from \eqref{eq:fixed-sector-aa-inverse}, the blocks
\eqref{eq:fixed-sector-h-blocks} and the first sum of
\eqref{eq:quaternion-sums}, we obtain \(P_{AA}^{-1}h_{A0}\eta\).
We apply \(h_{0A}\) using \eqref{eq:fixed-sector-h-blocks} and the same
quaternion sum once more, and subtract from
\(P_{00}\eta=-2\eta-2\bar\eta-\eta\hat q\).
The resulting coefficients are those in
\eqref{eq:fixed-sector-schur-coefficients}.
The determinant identity is the standard Schur-complement formula
\cite[\S8, (9), pp.~106--107]{cohen-2000-quaternionic}, together with
\(\det P_{AA}=N^4N_3^2\) from \eqref{eq:fixed-sector-aa-inverse}.
\end{proof}

We now write the Schur map as a diagonal matrix plus a skew matrix.
Since \(\Im\mathcal C=-\Im\mathcal D\) in
\eqref{eq:fixed-sector-schur-coefficients}, the real matrix of
\(\mathcal S\eta=\mathcal A\eta+\eta\mathcal B+
\mathcal C\bar\eta+\bar\eta\mathcal D\) in the basis \((1,u,v,z)\) is
\[
  \begin{pmatrix}
    d_0&-\mathbf w^{\mathsf T}\\
    \mathbf w&d_1I_3+[\mathbf k]_\times
  \end{pmatrix},
\]
where\([\mathbf k]_\times\boldsymbol\eta=\mathbf k\times\boldsymbol\eta\) and
\begin{align*}
  d_0&=\Re(\mathcal A+\mathcal B+\mathcal C+\mathcal D),
  &\quad d_1&=\Re(\mathcal A+\mathcal B-\mathcal C-\mathcal D),\\
  \mathbf w&=\Im(\mathcal A+\mathcal B),
  &\quad \mathbf k&=\Im(\mathcal A-\mathcal B-2\mathcal C).
\end{align*}
With \(\Lambda=\diag(d_0,d_1,d_1,d_1)\) and the skew matrix \(\Theta\)
built from \(\mathbf w,\mathbf k\), the matrix above is \(\Lambda+\Theta\), so
Lemma~\ref{lem:diagonal-skew-determinant} of Appendix~\ref{app:su2-bases}
gives
\begin{equation}
\label{eq:fixed-sector-schur-determinant}
 \det\nolimits_{\mathbb R}\mathcal S
   =d_0d_1^3+d_1^2|\mathbf w|^2+d_0d_1|\mathbf k|^2
       +(\mathbf w\cdot\mathbf k)^2.
\end{equation}

Since \(z\hat q\bar z=a_0-q_\perp+a_3z\), the real parts of
\(\zeta_1,\zeta_2\) are \(a_0/N,(a_0+3)/N_3\), and their imaginary
parts are \(\mathbf m/N,\mathbf m/N_3\), respectively. Substitution gives, on \(NN_3\ne0\),
\begin{equation}
\label{eq:fixed-sector-trace-norm-data}
\begin{aligned}
 d_0&=-4-a_0+\frac{12(a_0+3)}{N_3},
 &\text{and}\quad d_1&=-a_0+\frac{32a_0}{3N}+\frac{4(a_0+3)}{3N_3},\\
 \mathbf w&=-\mathbf p+\frac4{N_3}\mathbf m,
 &\text{and}\quad\mathbf k&=\mathbf p+\nu\mathbf m,\\
 \nu&=\tfrac43\Bigl(\frac4N-\frac1{N_3}\Bigr).
\end{aligned}
\end{equation}
Since \(|\mathbf p|^2=|\mathbf m|^2=x+y\) and \(\mathbf p\cdot\mathbf m=t\),
we find
\[
\begin{aligned}
  |\mathbf w|^2&=(x+y)\left(1+\frac{16}{N_3^2}\right)-\frac{8t}{N_3},\\
  |\mathbf k|^2&=(x+y)(1+\nu^2)+2\nu t,\\
  \mathbf w\cdot\mathbf k
    &=-(x+y)+\left(\frac4{N_3}-\nu\right)t+\frac{4\nu(x+y)}{N_3}.
\end{aligned}
\]
These are rational functions of \((a_0,x,y)\).

\subsection{The polynomial matrix and the form of
\texorpdfstring{$Q_{14}$}{Q14}}

\begin{lemma}
\label{lem:fixed-sector-polynomial-matrix}
Let \(\mathbf p,\mathbf m\) be as above, let
\(D=\diag(3,1,1,1)\) and \(\Delta=\diag(1,N,N,N)\).
Define the following matrices in the basis \((1,u,v,z)\):
\begin{gather*}
\mathcal S_{\mathrm{pol}}=\begin{pmatrix}-4-a_0&\mathbf p^{\mathsf T}\\-\mathbf p&-a_0I+[\mathbf p]_\times\end{pmatrix},\\
B_N=\begin{pmatrix}0&0\\0&2a_0I+[\mathbf m]_\times\end{pmatrix},\\
\Upsilon=\begin{pmatrix}a_0+3&-\mathbf m^{\mathsf T}\\ \mathbf m&(a_0+3)I-[\mathbf m]_\times\end{pmatrix},
\end{gather*}
where \(\mathcal S_{\mathrm{pol}}\) is the matrix of \(P_{00}\) and
\(\Upsilon\) is the matrix of right multiplication by
\(a_0+3+\mathbf m=\overline{3+z\hat q\bar z}\). Then, on \(NN_3\ne0\),
\begin{equation}
\label{eq:fixed-sector-pole-decomposition}
\begin{gathered}
\mathcal S=\mathcal S_{\mathrm{pol}}+\frac{16}{3N}\,B_N+\frac4{3N_3}\,D\Upsilon D,\\
\Upsilon\Upsilon^{\mathsf T}=N_3I\qquad\text{and}\qquad\det\Upsilon=N_3^2.
\end{gathered}
\end{equation}
Moreover the matrix
\begin{equation}
\label{eq:fixed-sector-matrix}
\begin{gathered}
M_{\mathrm{pol}}:=\Bigl(\Delta\mathcal S_{\mathrm{pol}}+\tfrac{16}3B_N\Bigr)\widetilde\Upsilon+\tfrac43\diag(9,N,N,N),\\
\widetilde\Upsilon=D^{-1}\Upsilon^{\mathsf T}D
\end{gathered}
\end{equation}
has entries in \(\mathbb Q[a_0,a_1,a_2,a_3]\), and
\(M_{\mathrm{pol}}=\Delta\,\mathcal S\,D^{-1}\Upsilon^{\mathsf T}D\) on \(NN_3\ne0\).
\end{lemma}

\begin{proof}
\emph{Split the poles.}
Split each entry of the data~\eqref{eq:fixed-sector-trace-norm-data}
into its polynomial part, its \(1/N\) part and its \(1/N_3\) part, using
\(\nu\mathbf m=\tfrac{16}{3N}\mathbf m-\tfrac4{3N_3}\mathbf m\). The
polynomial parts \(-4-a_0\), \(-a_0\), \(-\mathbf p\), \(\mathbf p\) of
\(d_0,d_1,\mathbf w,\mathbf k\) give \(\mathcal S_{\mathrm{pol}}\); only
\(d_1\) and \(\mathbf k\) have a \(1/N\) part, \(\tfrac{32a_0}{3N}\) and
\(\tfrac{16}{3N}\mathbf m\), which give \(\tfrac{16}{3N}B_N\); and the
\(1/N_3\) parts \(12(a_0+3)\), \(4\mathbf m\), \(\tfrac43(a_0+3)\),
\(-\tfrac43\mathbf m\) of \(d_0,\mathbf w,d_1,\mathbf k\), each over \(N_3\),
give \(\frac4{3N_3}D\Upsilon D\).

\emph{Evaluate the quaternionic factor.}
Write \(\Upsilon=(a_0+3)I+\Xi\) with
\(\Xi=\begin{pmatrix}0&-\mathbf m^{\mathsf T}\\ \mathbf m&-[\mathbf m]_\times\end{pmatrix}\)
skew. Then \(\Xi^2=-|\mathbf m|^2I\), so
\(\Upsilon\Upsilon^{\mathsf T}=\bigl((a_0+3)^2+x+y\bigr)I=N_3I\). Hence
\((\det\Upsilon)^2=N_3^4\), and \(\det\Upsilon=N_3^2\), because
\(\det\Upsilon-N_3^2\) and \(\det\Upsilon+N_3^2\) cannot both be nonzero in
the domain \(\mathbb Q[a_0,a_1,a_2,a_3]\) and the value at \(\mathbf m=0\) is
\((a_0+3)^4=N_3^2\).

\emph{Clear the denominators.}
By the pole decomposition,
\(\mathcal SD^{-1}\Upsilon^{\mathsf T}D
=\mathcal S_{\mathrm{pol}}\widetilde\Upsilon
+\tfrac{16}{3N}B_N\widetilde\Upsilon
+\tfrac4{3N_3}D\Upsilon\Upsilon^{\mathsf T}D\), and
\(D\Upsilon\Upsilon^{\mathsf T}D=N_3D^2\). The first row of \(B_N\)
vanishes, so \(\Delta B_N=NB_N\). Multiplying by \(\Delta\) on
the left gives \eqref{eq:fixed-sector-matrix}.
\end{proof}

\begin{proposition}
\label{prop:fixed-sector-q14-determinant}
Let \(M_{\mathrm{pol}}\) be the polynomial matrix of
Lemma~\ref{lem:fixed-sector-polynomial-matrix}, and define
\(Q_{14}:=\det M_{\mathrm{pol}}\in\mathbb Q[a_0,a_1,a_2,a_3]\).
Then, as identities of polynomials:
\begin{enumerate}
\item \(\det\bigl(P_{\avec}|_{W_{\mathbf1}}\bigr)=N\,\det M_{\mathrm{pol}}=N\,Q_{14}\);
\item on \(NN_3\neq0\),
\begin{equation}
\label{eq:fixed-sector-q14-closed-form}
 Q_{14}=N^3N_3^2
 \bigl(d_0d_1^3+d_1^2|\mathbf w|^2+d_0d_1|\mathbf k|^2
                 +(\mathbf w\cdot\mathbf k)^2\bigr),
\end{equation}
with the data of~\eqref{eq:fixed-sector-trace-norm-data}.
\end{enumerate}
\end{proposition}

\begin{proof}
Since \(\det\Delta=N^3\) and
\(\det(D^{-1}\Upsilon^{\mathsf T}D)=\det\Upsilon=N_3^2\)
(Lemma~\ref{lem:fixed-sector-polynomial-matrix}), we have
\(\det M_{\mathrm{pol}}=N^3N_3^2\det\mathcal S\) on \(NN_3\ne0\); with
\eqref{eq:fixed-sector-schur-determinant}, this is
\eqref{eq:fixed-sector-q14-closed-form}, which is item~2.
By Lemma~\ref{lem:fixed-sector-schur},
\begin{equation*}
\det\bigl(P_{\avec}|_{W_{\mathbf1}}\bigr)=N^4N_3^2\det\mathcal S
=N\det M_{\mathrm{pol}} \qquad \text{on} \qquad NN_3\ne0.
\end{equation*}
Since both sides are polynomials, and $NN_3\ne0$ is dense, item 1 holds identically.
\end{proof}

\bibliographystyle{habbrv}
\bibliography{references}

\end{document}